\documentclass[11pt,english,a4paper]{article}

\usepackage[T1]{fontenc}
\usepackage[latin9]{inputenc}
\usepackage{geometry}
\usepackage{babel}
\usepackage{amsmath}
\usepackage{amsthm}
\usepackage{amssymb}
\usepackage[unicode=true,pdfusetitle,
 bookmarks=true,bookmarksnumbered=false,bookmarksopen=false,
 breaklinks=false,pdfborder={0 0 0},pdfborderstyle={},backref=false,colorlinks=false]
 {hyperref}

\makeatletter
\theoremstyle{plain}
\newtheorem{thm}{\protect\theoremname}[section]
\theoremstyle{plain}
\newtheorem{cor}[thm]{\protect\corollaryname}
\theoremstyle{plain}
\newtheorem{lem}[thm]{\protect\lemmaname}
\theoremstyle{plain}
\newtheorem{prop}[thm]{\protect\propositionname}
\theoremstyle{plain}
\newtheorem{conjecture}[thm]{\protect\conjecturename}
\theoremstyle{definition}
\newtheorem{defn}[thm]{\protect\definitionname}
\theoremstyle{definition}
\newtheorem{example}[thm]{\protect\examplename}
\theoremstyle{remark}
\newtheorem{rem}[thm]{\protect\remarkname}
\theoremstyle{remark}
\newtheorem{claim}[thm]{\protect\claimname}
\theoremstyle{plain}
\newtheorem{fact}[thm]{\protect\factname}
\theoremstyle{definition}
\newtheorem{sol}[thm]{\protect\solutionname}

\date{}

\usepackage{appendix}

\usepackage{bm}

\usepackage[nameinlink,capitalise,noabbrev]{cleveref}
\AtBeginDocument{\renewcommand{\ref}[1]{\cref{#1}}}
\crefname{appsec}{Appendix}{Appendices}
\crefname{enumi}{condition}{conditions}
\Crefname{enumi}{Condition}{Conditions}

\theoremstyle{plain}
\newtheorem{mythm}{\protect\theoremname}[section]
\renewenvironment{thm}{\begin{mythm}}{\end{mythm}}
\crefname{mythm}{Theorem}{Theorems}
\theoremstyle{definition}
\newtheorem{mydefn}[mythm]{\protect\definitionname}
\renewenvironment{defn}{\begin{mydefn}}{\end{mydefn}}
\theoremstyle{definition}
\newtheorem{myfact}[mythm]{\protect\factname}
\renewenvironment{fact}{\begin{myfact}}{\end{myfact}}
\theoremstyle{definition}
\newtheorem{myexample}[mythm]{\protect\examplename}
\renewenvironment{example}{\begin{myexample}}{\end{myexample}}
\theoremstyle{plain}
\newtheorem{myprop}[mythm]{\protect\propositionname}
\renewenvironment{prop}{\begin{myprop}}{\end{myprop}}
\theoremstyle{plain}
\newtheorem{mycor}[mythm]{\protect\corollaryname}
\renewenvironment{cor}{\begin{mycor}}{\end{mycor}}
\theoremstyle{plain}
\newtheorem{mylem}[mythm]{\protect\lemmaname}
\renewenvironment{lem}{\begin{mylem}}{\end{mylem}}
\crefname{mylem}{Lemma}{Lemmas}
\theoremstyle{plain}
\newtheorem{myconjecture}{\protect\conjecturename}
\renewenvironment{conjecture}{\begin{myconjecture}}{\end{myconjecture}}
\theoremstyle{definition}
\newtheorem{myrem}[mythm]{\protect\remarkname}
\renewenvironment{rem}{\begin{myrem}}{\end{myrem}}
\theoremstyle{plain}
\newtheorem{myclaim}[mythm]{\protect\claimname}
\renewenvironment{claim}{\begin{myclaim}}{\end{myclaim}}
\theoremstyle{plain}
\newtheorem{myfakethm}[mythm]{``Theorem''}
\renewenvironment{sol}{\begin{myfakethm}}{\end{myfakethm}}
\crefformat{equation}{#2(#1)#3}

\let\originalleft\left
\let\originalright\right
\renewcommand{\left}{\mathopen{}\mathclose\bgroup\originalleft}
\renewcommand{\right}{\aftergroup\egroup\originalright}
\usepackage{pgfplots}
\usetikzlibrary{pgfplots.groupplots}
\usepackage{verbatim}

\makeatletter
\renewcommand*{\UrlTildeSpecial}{%
  \do\~{%
    \mbox{%
      \fontfamily{ptm}\selectfont
      \textasciitilde
    }%
  }%
}%
\let\Url@force@Tilde\UrlTildeSpecial
\makeatother

\usepackage{tikz}
\tikzstyle{vertex}=[circle,draw=black,fill=black,inner sep=0,minimum size=0.2cm,text=white,font=\footnotesize]
\tikzset{every loop/.style={min distance=50,in=50,out=130,looseness=7}}

\usepackage[labelfont=bf,labelsep=period]{caption}

\let\OLDthebibliography\thebibliography
\renewcommand\thebibliography[1]{
  \OLDthebibliography{#1}
  \setlength{\parskip}{0pt}
  \setlength{\itemsep}{3pt plus 0.3ex}
}

\makeatother

\providecommand{\claimname}{Claim}
\providecommand{\conjecturename}{Conjecture}
\providecommand{\corollaryname}{Corollary}
\providecommand{\definitionname}{Definition}
\providecommand{\examplename}{Example}
\providecommand{\factname}{Fact}
\providecommand{\lemmaname}{Lemma}
\providecommand{\propositionname}{Proposition}
\providecommand{\remarkname}{Remark}
\providecommand{\solutionname}{Solution}
\providecommand{\theoremname}{Theorem}

\begin{document}
\title{Almost all Steiner triple systems have perfect matchings}
\author{Matthew Kwan \thanks{Department of Mathematics, Stanford University, Stanford, CA 94305.
Email: \href{mattkwan@stanford.edu}{\nolinkurl{mattkwan@stanford.edu}}.
This research was done while the author was working at ETH Zurich,
and is supported in part by SNSF project 178493.}}

\maketitle
\global\long\def\RR{\mathbb{R}}%

\global\long\def\QQ{\mathbb{Q}}%

\global\long\def\HH{\mathbb{H}}%

\global\long\def\E{\mathbb{E}}%

\global\long\def\Var{\operatorname{Var}}%

\global\long\def\CC{\mathbb{C}}%

\global\long\def\NN{\mathbb{N}}%

\global\long\def\ZZ{\mathbb{Z}}%

\global\long\def\GG{\mathbb{G}}%

\global\long\def\BB{\mathbb{B}}%

\global\long\def\DD{\mathbb{D}}%

\global\long\def\cL{\mathcal{L}}%

\global\long\def\supp{\operatorname{supp}}%

\global\long\def\one{\boldsymbol{1}}%

\global\long\def\range#1{\left[#1\right]}%

\global\long\def\d{\operatorname{d}\!}%

\global\long\def\falling#1#2{\left(#1\right)_{#2}}%

\global\long\def\f{\mathbf{f}}%

\global\long\def\im{\operatorname{im}}%

\global\long\def\sp{\operatorname{span}}%

\global\long\def\rank{\operatorname{rank}}%

\global\long\def\sign{\operatorname{sign}}%

\global\long\def\mod{\operatorname{mod}}%

\global\long\def\id{\operatorname{id}}%

\global\long\def\disc{\operatorname{disc}}%

\global\long\def\lindisc{\operatorname{lindisc}}%

\global\long\def\tr{\operatorname{tr}}%

\global\long\def\adj{\operatorname{adj}}%

\global\long\def\Unif{\operatorname{Unif}}%

\global\long\def\Po{\operatorname{Po}}%

\global\long\def\Bin{\operatorname{Bin}}%

\global\long\def\Ber{\operatorname{Ber}}%

\global\long\def\Geom{\operatorname{Geom}}%

\global\long\def\sat{\operatorname{sat}}%

\global\long\def\Hom{\operatorname{Hom}}%

\global\long\def\vol{\operatorname{vol}}%

\global\long\def\floor#1{\left\lfloor #1\right\rfloor }%

\global\long\def\ceil#1{\left\lceil #1\right\rceil }%

\global\long\def\cond{\,\middle|\,}%

\let\polishL\L

\global\long\def\L{\mathcal{L}}%

\DeclareRobustCommand{\L}{\ifmmode{\mathcal{L}}\else\polishL\fi}

\global\long\def\randS{\boldsymbol{S}}%

\global\long\def\S{\mathcal{S}}%

\global\long\def\Sm#1{\mathcal{S}_{#1}}%

\global\long\def\ord{\mathcal{O}}%

\global\long\def\ordm#1{\mathcal{O}_{#1}}%

\global\long\def\cordm#1#2{\mathcal{O}_{#2}^{#1}}%

\global\long\def\cSm#1#2{\mathcal{S}_{#2}^{#1}}%

\global\long\def\ext#1{\mathcal{O}^{*}\left(#1\right)}%

\global\long\def\randN{\boldsymbol{N}}%

\global\long\def\randX{\boldsymbol{X}}%

\global\long\def\rg#1#2{\RR\left(#1,#2\right)}%

\global\long\def\Gnp#1#2{\GG\left(#1,#2\right)}%

\global\long\def\GSnp#1#2{\GG^{*}\left(#1,#2\right)}%

\global\long\def\randG{\boldsymbol{G}}%

\global\long\def\randY{\boldsymbol{Y}}%

\global\long\def\randZ{\boldsymbol{Z}}%

\global\long\def\randM{\boldsymbol{M}}%

\global\long\def\randL{\boldsymbol{L}}%

\begin{abstract}
We show that for any $n$ divisible by 3, almost all order-$n$ Steiner
triple systems have a perfect matching (also known as a \emph{parallel
class} or \emph{resolution class}). In fact, we prove a general upper
bound on the number of perfect matchings in a Steiner triple system
and show that almost all Steiner triple systems essentially attain
this maximum. We accomplish this via a general theorem comparing a
uniformly random Steiner triple system to the outcome of the triangle
removal process, which we hope will be useful for other problems.
Our methods can also be adapted to other types of designs; for example,
we sketch a proof of the theorem that almost all Latin squares have
transversals.
\end{abstract}

\section{Introduction}

A \emph{Steiner triple system} of order $n$ is a collection $S$
of size-$3$ subsets of $\range n=\left\{ 1,\dots,n\right\} $ (that
is, a \emph{$3$-uniform hypergraph} on the vertex set $\range n$),
such that every pair of vertices is included in exactly one hyperedge
of $S$. Steiner triple systems are among the most fundamental types
of combinatorial designs, and have strong connections to a wide range
of different subjects, ranging from group theory, to finite geometry,
to experimental design, to the theory of error-correcting codes. See
\cite{CR99} for an introduction to the subject. Observe that a Steiner
triple system is actually nothing more than a partition of the edges
of the complete graph $K_{n}$ into triangles (a \emph{triangle-decomposition}
of $K_{n}$), so Steiner triple systems are natural ``symmetric''
counterparts to \emph{Latin squares}, which can be defined as triangle-decompositions
of the complete tripartite graph $K_{n,n,n}$.

In 1974 Wilson \cite{Wil74} used estimates for the number of Latin
squares to prove a coarse estimate for the number of Steiner triple
systems. Babai \cite{Bab80} used this estimate to prove that almost
all Steiner triple systems have trivial automorphism group (that is
to say, a uniformly random order-$n$ Steiner triple system a.a.s.\footnote{By ``asymptotically almost surely'', or ``a.a.s.'', we mean that
the probability of an event is $1-o\left(1\right)$. Here and for
the rest of the paper, asymptotics are as $n\to\infty$.} has trivial automorphism group). We believe this is the only nontrivial
property known to hold a.a.s. for random Steiner triple systems. Following
Erd\H os and R\'enyi's seminal paper \cite{ER59} on random graphs
and Erd\H os' popularization of the probabilistic method, there have
been great developments in the theory of random combinatorial structures
of all kinds, but essentially none of the tools developed seem to
be applicable to Steiner triple systems. Steiner triple systems lack
independence or any kind of recursive structure, which rules out many
of the techniques used to study Erd\H os\textendash R\'enyi random
graphs and random permutations, and there is basically no freedom
to make local changes, which precludes the use of ``switching''
techniques often used in the study of random regular graphs (see for
example \cite{KSVW01}). It is not even clear how to study random
Steiner triple systems empirically; in an attempt to find an efficient
algorithm to generate a random Steiner triple system, Cameron \cite{Cam02}
designed a Markov chain on Steiner triple systems, but he was not
able to determine whether this chain was connected.

In a recent breakthrough, Keevash \cite{Kee14} proved that for a
large class of combinatorial designs generalising Steiner triple systems,
``partial'' designs satisfying a certain ``quasirandomness'' condition
can be completed into designs. Shortly afterwards \cite{Kee15}, he
showed that his results could be used for approximate enumeration;
in particular, matching an upper bound due to Linial and Luria \cite{LL13}
he proved that there are
\begin{equation}
\left(n/e^{2}+o\left(n\right)\right)^{n^{2}/6}\label{eq:num-STS}
\end{equation}
Steiner triple systems of order $n$, as long as $n$ satisfies a
necessary divisibility condition (Steiner triple systems can only
exist if $n$ is 1 or 3 mod 6).

Of course, this new estimate makes it possible, in theory, to prove
new a.a.s. properties of random Steiner triple systems just by giving
an estimate asymptotically smaller than \ref{eq:num-STS} for the
number of Steiner triple systems not satisfying a certain property.
However, for most properties it is not at all clear how to prove such
estimates. Instead, we introduce a way to use Keevash's methods to
show that a uniformly random Steiner triple system can in some sense
be approximated by the outcome of a random process called the \emph{triangle
removal process}. We remark that actually Keevash proved \ref{eq:num-STS}
with a randomised construction that involves the triangle removal
process, so many properties that hold a.a.s. in the triangle removal
process trivially hold a.a.s. in this random construction. Such results
have been proved in \cite[Proposition~3.1]{LL15} and \cite{LLR15}.
However, the Steiner triple systems obtainable by Keevash's construction
comprise a negligible proportion of the set of Steiner triple systems,
and a somewhat more delicate approach is required to study a uniformly
random Steiner triple system. In \ref{sec:random-STS} we state a
general theorem summarising our method.

A \emph{matching }in a hypergraph is a collection of disjoint edges,
and a \emph{perfect matching }is a matching covering the entire vertex
set. The existence of perfect matchings is one of the most central
questions in the theory of graphs and hypergraphs; in particular,
some of the most important recent developments in the field are the
work of Johansson, Kahn and Vu \cite{JKV08} on perfect matchings
in random hypergraphs, and the theory of Keevash and Mycroft \cite{KM15}
characterising when dense hypergraphs have perfect matchings. We should
also mention the ``nibble'' method of R\"odl \cite{Rod85} for
finding almost-perfect hypergraph matchings, which has had a significant
influence throughout combinatorics in the last 30 years. A perfect
matching in a Steiner triple system is also called a\emph{ parallel
class} or \emph{resolution class}, and has particular significance.
One of the oldest problems in combinatorics, famously solved in the
affirmative by Ray-Chaudhuri and Wilson \cite{RW71}, asks whether
for all $n\equiv3\mod6$ there exists an order-$n$  Steiner triple
system which can be partitioned into hyperedge-disjoint perfect matchings
(a \emph{Kirkman triple system}). Alon, Kim and Spencer \cite{AKS97}
proved that every Steiner triple system has an almost-perfect matching
covering all but $o\left(\sqrt{n}\log^{3/2}n\right)$ vertices, and
Bryant and Horsley \cite{BH15} proved that for infinitely many $n\equiv3\mod6$
there exist Steiner triple systems with no perfect matching. As an
application of our new method, we prove our main result that if $n\equiv3\mod6$
(that is, if $3\mid n$ and an order-$n$ Steiner triple system exists)
then almost all order-$n$ Steiner triple systems have many perfect
matchings. 
\begin{thm}
\label{thm:PM-in-STS}Let $n\equiv3\mod6$ and let $\randS$ be a
uniformly random order-$n$ Steiner triple system. Then a.a.s. $\randS$
contains
\[
\left(\left(1-o\left(1\right)\right)\frac{n}{2e^{2}}\right)^{n/3}
\]
perfect matchings.
\end{thm}

We remark that if $n\equiv1\mod6$ then obviously no order-$n$ Steiner
triple system can have a perfect matching, but exactly the same proof
can be used to show that in a random order-$n$ Steiner triple system
there is a.a.s. a matching covering all but one vertex.

We prove \ref{thm:PM-in-STS} using our new method combined with the
so-called \emph{absorbing method}, which was introduced as a general
method by R\"odl, Ruci\'nski and Szemer\'edi \cite{RRS09} (the
basic idea had been used earlier, for example by Krivelevich \cite{Kri97}).
Basically, we prove the a.a.s. existence of certain substructures
that are ``flexible'' and allow us to complete an almost-perfect
matching into a perfect one.

Up to the error term, a random Steiner triple system actually has
the maximum possible number of perfect matchings: we also prove the
following upper bound.
\begin{thm}
\label{thm:maximum-PMs}Any Steiner triple system $S$ has at most
\[
\left(\left(1+o\left(1\right)\right)\frac{n}{2e^{2}}\right)^{n/3}
\]
perfect matchings.
\end{thm}

The proof of \ref{thm:maximum-PMs} is quite short, and uses the notion
of \emph{entropy}. This particular type of argument was introduced
by Radhakrishnan \cite{Rad97} and further developed by Linial and
Luria \cite{LL13}, among others.

\subsection{Latin squares}

An order-$n$ Latin square is usually defined as an $n\times n$ array
of the numbers between 1 and $n$ (we call these \emph{symbols}),
such that each row and column contains each symbol exactly once. As
mentioned earlier, this is equivalent to a 3-uniform hypergraph whose
hyperedges comprise a triangle-decomposition of the edges of the complete
tripartite graph $K_{n,n,n}$ (the three parts correspond to the rows,
columns and symbols, so a triangle $\left(i,j,k\right)$ corresponds
to putting the symbol $k$ in the cell $\left(i,j\right)$). A perfect
matching in this hypergraph is called a \emph{transversal }and the
property of containing a transversal is of great interest. In particular,
the famous Ryser\textendash Brualdi\textendash Stein conjecture speculates
that every odd-order Latin square has a transversal, and every even-order
Latin square has a partial transversal of size $n-1$. Although this
conjecture remains wide open, there have been many partial results
and generalizations; see for example \cite{ES91,CW05,HS08,AB09,Pok16}.
See \cite{Wan11} for an introduction to the subject of Latin transversals,
and \cite{DK15} for an introduction to Latin squares in general.

The counterpart of \ref{thm:maximum-PMs} for Latin squares, that
a Latin square can have no more than $\left(\left(1+o\left(1\right)\right)n/e^{2}\right)^{n}$
transversals, was first proved by Taranenko \cite{Tar15}. Glebov
and Luria \cite{GL16} gave a simpler entropy-based proof of the same
fact and asked whether the counterpart of \ref{thm:PM-in-STS} holds:
do almost all Latin squares have essentially the maximum possible
number of transversals? Although there do exist a number of techniques
for studying random Latin squares (see for example \cite{Cam92,HJ96,MW99,CGW08,KS16}),
none of them seem suitable to attack this question.

In the time since the first version of this paper, Keevash \cite{Kee18}
has generalised his methods to a number of different classes of designs,
including Latin squares. Using a result from this recent work, it
is straightforward to adapt the methods used to prove \ref{thm:PM-in-STS}
to prove the following theorem, answering Glebov and Luria's question
and proving that the Ryser\textendash Brualdi\textendash Stein conjecture
holds for almost all Latin squares.
\begin{thm}
\label{thm:latin}Let $\randL$ be a uniformly random order-$n$ Latin
square. Then a.a.s. $\randL$ contains
\[
\left(\left(1-o\left(1\right)\right)\frac{n}{e^{2}}\right)^{n}
\]
transversals.
\end{thm}

\subsection{Structure of the paper}

The structure of this paper is as follows. In \ref{sec:random-STS}
we present our general theorem for comparing random Steiner triple
systems with the triangle removal process. The proof requires a straightforward
but necessary generalization of \ref{eq:num-STS}, estimating the
number of completions of a partial Steiner triple system, and this
in turn requires a routine analysis of the triangle removal process.
These parts of the proof are deferred to \ref{sec:completion-count}
and \ref{sec:random-triangle-removal}. In \ref{sec:absorbers} we
use the theory from \ref{sec:random-STS} and the absorbing method
to prove \ref{thm:PM-in-STS}. In \ref{sec:max-num-matchings} we
prove \ref{thm:maximum-PMs}, and in \ref{sec:latin} we explain how
to adapt our methods to prove \ref{thm:latin}. Finally, in \ref{sec:concluding}
we have some concluding remarks, including a long list of open problems.

\subsection{Notation}

We use standard asymptotic notation throughout. For functions $f=f\left(n\right)$
and $g=g\left(n\right)$:
\begin{itemize}
\item $f=O\left(g\right)$ means there is a constant $C$ such that $\left|f\right|\le C\left|g\right|$,
\item $f=\Omega\left(g\right)$ means there is a constant $c>0$ such that
$f\ge c\left|g\right|$,
\item $f=\Theta\left(g\right)$ means that $f=O\left(g\right)$ and $f=\Omega\left(g\right)$,
\item $f=o\left(g\right)$ means that $f/g\to0$ as $n\to\infty$. To say
that a.a.s. $f=o\left(g\right)$ means that for any $\varepsilon>0$,
a.a.s. $f/g<\varepsilon$.
\end{itemize}
Also, following \cite{Kee15}, the notation $f=1\pm\varepsilon$ means
$1-\varepsilon\le f\le1+\varepsilon$.

We also use standard graph theory notation: $V\left(G\right)$ and
$E\left(G\right)$ are the sets of vertices and (hyper)edges of a
(hyper)graph $G$, and $v\left(G\right)$ and $e\left(G\right)$ are
the cardinalities of these sets. The subgraph of $G$ induced by a
vertex subset $U$ is denoted $G\left[U\right]$, the degree of a
vertex $v$ is denoted $\deg_{G}\left(v\right)$, and the subgraph
obtained by deleting $v$ is denoted $G-v$.

For a positive integer $n$, we write $\range n$ for the set $\left\{ 1,2,\dots,n\right\} $.
For a real number $x$, the floor and ceiling functions are denoted
$\floor x=\max\left\{ i\in\ZZ:i\le x\right\} $ and $\ceil x=\min\left\{ i\in\ZZ:i\ge x\right\} $.
We will however mostly omit floor and ceiling signs and assume large
numbers are integers, wherever divisibility considerations are not
important. All logarithms are in base $e$.

Finally, we remark that throughout the paper we adopt the convention
that random variables (and random objects more generally) are printed
in bold.

\subsection{Acknowledgements}

The author would like to thank Asaf Ferber and Benny Sudakov for very
helpful discussions about random designs and the intricacies of the
absorbing method, Patrick Morris and Stefan Glock for pointing out
some oversights in earlier versions of this paper, and Lutz Warnke
for pointing out some related references.

\section{\label{sec:random-STS}Random Steiner triple systems via the triangle
removal process}

In this section we describe our method for comparing random Steiner
triple systems with the outcome of the triangle removal process. Define
\[
N={n \choose 2}/3=\left(1+o\left(1\right)\right)n^{2}/6
\]
to be the number of hyperedges in a Steiner triple system. We assume
throughout this section that $n$ is 1 or 3 mod 6.
\begin{defn}[partial systems]
A \emph{partial Steiner triple system} (or \emph{partial system}
for short) is a $3$-uniform hypergraph on $\range n$ in which every
pair of vertices is included in no more than one hyperedge. Let $\Sm m$
be the set of partial systems with $m$ hyperedges. We will also want
to consider partial systems equipped with an ordering on their hyperedges.
Let $\ord$ be the set of ordered Steiner triple systems, and let
$\ordm m$ be the set of ordered partial systems with $m$ hyperedges.
For $S\in\ordm m$ and $i\le m$, let $S_{i}$ be the ordered partial
system consisting of just the first $i$ hyperedges of $S$. For a
(possibly ordered) partial system $S$, let $G\left(S\right)$ be
the graph with an edge for every pair of vertices which does not\emph{
}appear in any hyperedge of $S$. So, if $S$ has $m$ hyperedges,
then $G\left(S\right)$ has ${n \choose 2}-3m$ edges.
\end{defn}

\begin{defn}[quasirandomness]
For a graph $G$ with $n$ vertices and $m$ edges, let $d\left(G\right)=m/{n \choose 2}$
denote its density. We say $G$ is \emph{$\left(\varepsilon,h\right)$-quasirandom}
if for every set $A$ of at most $h$ vertices, we have $\left|\bigcap_{w\in A}N_{G}\left(w\right)\right|=\left(1\pm\varepsilon\right)d\left(G\right)^{\left|A\right|}n$.
Let $\cSm{\varepsilon,h}m\subseteq\Sm m$ be the set of partial systems
$S\in\Sm m$ such that $G\left(S\right)$ is $\left(\varepsilon,h\right)$-quasirandom,
and let $\cordm{\varepsilon,h}m\subseteq\ordm m$ be the set of ordered
partial systems $S\in\ordm m$ such that $S_{i}\in\cSm{\varepsilon,h}i$
for each $i\le m$.
\end{defn}

\begin{defn}[the triangle removal process]
The triangle removal process is defined as follows. Start with the
complete graph $K_{n}$ and iteratively delete a triangle chosen uniformly
at random from all triangles in the remaining graph. If we continue
this process for $m$ steps, the deleted triangles (in order) can
be interpreted as an ordered partial system in $\ordm m$. It is also
possible that the process aborts (because there are no triangles left)
before $m$ steps, in which case we say it returns the value ``$*$''.
We denote by $\rg nm$ the resulting distribution on $\ordm m\cup\left\{ *\right\} $.
\end{defn}

Now, we can state a general theorem comparing random Steiner triple
systems with the triangle removal process. Basically, if we can show
that the first few edges of the triangle removal process (as an ordered
partial system) satisfy some property with extremely high probability,
then it follows that the first few edges of a uniformly random ordered
Steiner triple system satisfy the same property with high probability.
Moreover, it suffices to study the triangle removal process conditioned
on some ``good'' event, provided that this event contains the event
that our partial system is sufficiently quasirandom.
\begin{thm}
\label{lem:triangle-removal-transfer}Fixing $h\in\NN$ and sufficiently
small $a>0$, there is $b=b\left(a,h\right)>0$ such that the following
holds. Fix $\alpha\in\left(0,1\right)$, let $\mathcal{P}\subseteq\ordm{\alpha N}$
be a property of ordered partial systems, let $\varepsilon=n^{-a}$,
 let $\mathcal{Q}\supseteq\cordm{\varepsilon,h}{\alpha N}$, let $\randS\in\ord$
be a uniformly random ordered Steiner triple system and let $\randS'\sim\rg n{\alpha N}$.
If
\[
\Pr\left(\randS'\notin\mathcal{P}\cond\randS'\in\mathcal{Q}\right)\le\exp\left(-n^{2-b}\right)
\]
then
\[
\Pr\left(\randS_{\alpha N}\notin\mathcal{P}\right)\le\exp\left(-\Omega\left(n^{1-2a}\right)\right).
\]
\end{thm}

Note that (as we prove in \ref{sec:random-triangle-removal}), the
triangle removal process is likely to produce quasirandom graphs;
that is, $\Pr\left(\randS'\in\mathcal{Q}\right)\ge\Pr\left(\randS'\in\cordm{\varepsilon,h}{\alpha N}\right)=1-o\left(1\right)$.
However, as we will see in \ref{subsec:absorbers}, the conditioning
in \ref{lem:triangle-removal-transfer} can still be useful because
the probabilities under consideration are so small (it is certainly
not true that $\Pr\left(\randS'\notin\mathcal{Q}\right)$ is anywhere
near as small as $\exp\left(-\Omega\left(n^{2}\right)\right)$).

The proof of \ref{lem:triangle-removal-transfer} follows from a sequence
of several lemmas. The most important is the following: we can estimate
the number of ways to complete a partial system $S$, and show that
it does not vary too much between choices of $S$.
\begin{lem}
\label{lem:num-extensions}For an ordered partial system $S\in\ordm m$,
let $\ext S\subseteq\ord$ be the set of ordered Steiner triple systems
$S^{*}$ such that $S_{m}^{*}=S$. Fixing sufficiently large $h\in\NN$
and any $a>0$, there is $b=b\left(a,h\right)>0$ such that the following
holds. For any fixed $\alpha\in\left(0,1\right)$, any $\varepsilon=\varepsilon\left(n\right)\le n^{-a}$
and any $S,S'\in\cordm{\varepsilon,h}{\alpha N}$,
\[
\frac{\left|\ext S\right|}{\left|\ext{S'}\right|}\le\exp\left(O\left(n^{2-b}\right)\right).
\]
\end{lem}

\ref{lem:num-extensions} can be proved with slight adaptations to
proofs of Keevash \cite{Kee15} and Linial and Luria \cite{LL13}
giving lower and upper bounds on the total number of Steiner triple
systems. The details are in \ref{sec:completion-count}.

The point of \ref{lem:num-extensions} is that if we can prove some
property holds with extremely high probability (say $1-\exp\left(-\Omega\left(n^{2}\right)\right)$)
in a uniformly random $\randS\in\cordm{\varepsilon,h}{\alpha N}$,
then it also holds with essentially the same probability in $\randS_{\alpha N}$,
for a uniformly random $\randS\in\ord$ conditioned on the event $\randS_{\alpha N}\in\cordm{\varepsilon,h}{\alpha N}$.
The next step is to show that the event $\randS_{\alpha N}\in\cordm{\varepsilon,h}{\alpha N}$
is very likely. In fact, this event occurs a.a.s. for a random ordering
of any given Steiner triple system. We prove the following lemma in
\ref{sec:randomly-ordered}.
\begin{lem}
\label{lem:random-is-typical}The following holds for any fixed $h\in\NN$,
$\alpha\in\left(0,1\right)$ and $a\in\left(0,1/2\right)$. Let $\varepsilon=n^{-a}$,
consider any Steiner triple system $S$, and uniformly at random order
its hyperedges to obtain an ordered Steiner triple system $\randS\in\ord$.
Then $\Pr\left(\randS_{\alpha N}\notin\cordm{\varepsilon,h}{\alpha N}\right)=\exp\left(-\Omega\left(n^{1-2a}\right)\right)$.
\end{lem}

The upshot of \ref{lem:num-extensions,lem:random-is-typical} is that
if we can prove a property holds with extremely high probability in
a uniformly random $\randS\in\cordm{\varepsilon,h}{\alpha N}$ for
sufficiently small $\varepsilon$ and sufficiently large $h$, then
that property also holds a.a.s. in the first $\alpha N$ hyperedges
of a uniformly random $\randS\in\ord$.

Next, the following lemma says that each $S\in\cordm{\varepsilon,h}{\alpha N}$
is roughly equally likely to be produced by the triangle removal process,
so that $\RR\left(n,\alpha N\right)$ approximates the uniform distribution
on $\cordm{\varepsilon,h}{\alpha N}$. It is proved in \ref{sec:greedy-random-approx-uniform}.
\begin{lem}
\label{lem:greedy-random-approx-uniform}The following holds for any
fixed $a\in\left(0,2\right)$ and $\alpha\in\left[0,1\right]$. Let
$\varepsilon=n^{-a}$, let $S,S'\in\cordm{\varepsilon,2}{\alpha N}$
and let $\randS\sim\rg n{\alpha N}$. Then
\[
\frac{\Pr\left(\randS=S\right)}{\Pr\left(\randS=S'\right)}\le\exp\left(O\left(n^{2-a}\right)\right).
\]
\end{lem}

We can finally combine everything to prove \ref{lem:triangle-removal-transfer}.
\begin{proof}[Proof of \ref{lem:triangle-removal-transfer}]
First, note that we can assume $h$ is large enough for \ref{lem:num-extensions},
because if we increase $h$ we still have $\mathcal{Q}\supseteq\cordm{\varepsilon,h}{\alpha N}$.
Let $\randS''\in\cordm{\varepsilon,h}{\alpha N}$ be a uniformly random
partial system in $\cordm{\varepsilon,h}{\alpha N}$. By \ref{lem:greedy-random-approx-uniform},
we have
\[
\Pr\left(\randS'\notin\mathcal{P}\cond\randS'\in\cordm{\varepsilon,h}{\alpha N}\right)=\exp\left(O\left(n^{2-a}\right)\right)\Pr\left(\randS''\notin\mathcal{P}\right).
\]
(Recalling the definition of big-oh notation, we emphasise that this
formula encapsulates a lower bound as well as an upper bound). Next,
let $c=b\left(a,h\right)$ in the notation of \ref{lem:num-extensions}.
We similarly have
\[
\Pr\left(\randS_{\alpha N}\notin\mathcal{P}\cond\randS_{\alpha N}\in\cordm{\varepsilon,h}{\alpha N}\right)=\exp\left(O\left(n^{2-c}\right)\right)\Pr\left(\randS''\notin\mathcal{P}\right).
\]
Using \ref{lem:random-is-typical}, it follows that
\begin{align*}
\Pr\left(\randS_{\alpha N}\notin\mathcal{P}\right) & \le\Pr\left(\randS_{\alpha N}\notin\mathcal{P}\cond\randS_{\alpha N}\in\cordm{\varepsilon,h}{\alpha N}\right)+\Pr\left(\randS_{\alpha N}\notin\cordm{\varepsilon,h}{\alpha N}\right)\\
 & \le\exp\left(O\left(n^{2-c}\right)\right)\exp\left(O\left(n^{2-a}\right)\right)\Pr\left(\randS'\notin\mathcal{P}\cond\randS'\in\cordm{\varepsilon,h}{\alpha N}\right)+\exp\left(-\Omega\left(n^{1-2a}\right)\right).
\end{align*}
But, if $a$ is small enough then $\Pr\left(\randS'\in\mathcal{Q}\right)\ge\Pr\left(\randS'\in\cordm{\varepsilon,h}{\alpha N}\right)=1-o\left(1\right)$,
by a standard analysis of the triangle removal process; see \ref{lem:triangle-removal-analysis}.
So,
\begin{align*}
\Pr\left(\randS'\notin\mathcal{P}\cond\randS'\in\cordm{\varepsilon,h}{\alpha N}\right) & =\frac{\Pr\left(\randS'\notin\mathcal{P}\text{ and }\randS'\in\cordm{\varepsilon,h}{\alpha N}\right)}{\Pr\left(\randS'\in\cordm{\varepsilon,h}{\alpha N}\right)}\\
 & \le\frac{\Pr\left(\randS'\notin\mathcal{P}\text{ and }\randS'\in\mathcal{Q}\right)}{\Pr\left(\randS'\in\cordm{\varepsilon,h}{\alpha N}\right)}\\
 & =\left(1+o\left(1\right)\right)\frac{\Pr\left(\randS'\notin\mathcal{P}\text{ and }\randS'\in\mathcal{Q}\right)}{\Pr\left(\randS'\in\mathcal{Q}\right)}\\
 & =\left(1+o\left(1\right)\right)\Pr\left(\randS'\notin\mathcal{P}\cond\randS'\in\mathcal{Q}\right).
\end{align*}
Choosing $b$ such that $b<\min\left\{ c,a\right\} $ and $2-b\ge1-2a$,
we then have
\[
\Pr\left(\randS_{\alpha N}\notin\mathcal{P}\right)\le\exp\left(-\Omega\left(n^{1-2a}\right)\right)
\]
as desired.
\end{proof}
In \ref{sec:randomly-ordered} we prove \ref{lem:random-is-typical}
and in \ref{sec:greedy-random-approx-uniform} we prove \ref{lem:greedy-random-approx-uniform}.
Also, in \ref{sec:nibble-TRP-coupling} we prove some lemmas which
are useful tools for applying \ref{lem:triangle-removal-transfer}
in practice.

\subsection{Randomly ordered Steiner triple systems\label{sec:randomly-ordered}}

In this subsection we prove \ref{lem:random-is-typical}.
\begin{proof}
Consider $m\le\alpha N$. Note that $\randS_{m}$ (as an unordered
partial system) is a uniformly random subset of $m$ hyperedges of
$S$. Also note that 
\[
d\left(G\left(\randS_{m}\right)\right)=\frac{\binom{n}{2}-3m}{\binom{n}{2}}=1-\frac{m}{N}.
\]
We can obtain a random partial system almost equivalent to $\randS_{m}$
by including each hyperedge of $S$ with independent probability $m/N$.
Let $\randS'$ denote the partial system so obtained, and let $\randG'=G\left(\randS'\right)$.
Now, fix a set $A$ of at most $h$ vertices. It suffices to prove
\begin{align}
\left|\bigcap_{w\in A}N_{\randG'}\left(w\right)\right| & =\left(1\pm n^{-a}\right)\left(1-\frac{m}{N}\right)^{\left|A\right|}n,\label{eq:degree-random-binomial}
\end{align}
with probability $1-\exp\left(-\Omega\left(n^{1-2a}\right)\right)$.
Indeed, the so-called Pittel inequality (see \cite[p.~17]{JLR00})
would imply that the same estimate holds with essentially the same
probability if we replace $\randS'$ with $\randS_{m}$ (thereby replacing
$\randG'$ with $G\left(\randS_{m}\right)$). We would then be able
to finish the proof by applying the union bound over all $m\le\alpha N$
and all choices of $A$.

Note that there are at most ${\left|A\right| \choose 2}=O\left(1\right)$
hyperedges of $S$ that include more than one vertex in $A$ (by the
defining property of a Steiner triple system). Let $U$ be the set
of vertices involved in these atypical hyperedges, plus the vertices
in $A$, so that $\left|U\right|=O\left(1\right)$. Let $\randN=\left|\left(\bigcap_{w\in A}N_{\randG'}\left(w\right)\right)\backslash U\right|$.
For every $v\notin U$ and $w\in A$ there is exactly one hyperedge
$e_{v}^{w}$ in $S$ containing $v$ and $w$, whose presence in $\randS'$
would prevent $v$ from contributing to $\randN$. For each fixed
$v\notin U$ the hyperedges $e_{v}^{w}$, for $w\in A$, are distinct,
so
\[
\Pr\left(v\in\bigcap_{w\in A}N_{\randG'}\left(w\right)\right)=\left(1-\frac{m}{N}\right)^{\left|A\right|},
\]
and by linearity of expectation $\E\randN=\left(1-m/N\right)^{\left|A\right|}\left(n-O\left(1\right)\right)$.
Now, $\randN$ is determined by the presence of at most $\left(n-\left|U\right|\right)\left|A\right|=O\left(n\right)$
hyperedges in $\randS'$, and changing the presence of each affects
$\randN$ by at most $2=O\left(1\right)$. So, by the Azuma\textendash Hoeffding
inequality (see \cite[Section~2.4]{JLR00}),
\begin{align*}
\Pr\left(\left|\randN-\left(1-\frac{m}{N}\right)^{\left|A\right|}n\right|>n^{-a}\left(1-\frac{m}{N}\right)^{\left|A\right|}n-\left|U\right|\right) & \le\exp\left(-\Omega\left(\frac{\left(n^{-a}\left(1-\alpha\right)^{h}n\right)^{2}}{n}\right)\right)\\
 & =\exp\left(-\Omega\left(n^{1-2a}\right)\right).
\end{align*}
Finally, we recall that $\left|\left(\bigcap_{w\in A}N_{\randG'}\left(w\right)\right)\right|=\randN\pm\left|U\right|$,
which completes the proof of \ref{eq:degree-random-binomial}.
\end{proof}

\subsection{Approximate uniformity of the triangle removal process\label{sec:greedy-random-approx-uniform}}

In this subsection we prove \ref{lem:greedy-random-approx-uniform}.
We first make the simple observation that subgraph statistics in a
quasirandom graph $G$ can be easily estimated in terms of the density
of $G$. In fact we prove a more general extension lemma about ``rooted''
subgraphs, which we will also use later in the paper. An \emph{embedding}
of a graph $H$ in a graph $G$ is a surjective homomorphism from
$H$ into $G$.
\begin{prop}
\label{prop:H-count}Let $H$ be a fixed graph with identified vertices
$u_{1},\dots,u_{k}$. Let $G$ be an $\left(\varepsilon,v\left(H\right)-1\right)$-quasirandom
graph on $n$ vertices, and let $\phi:\left\{ u_{1},\dots,u_{k}\right\} \to G$
be an embedding of $F=H\left[\left\{ u_{1},\dots,u_{k}\right\} \right]$
in $G$. Then, the number of ways to extend $\phi$ to an embedding
of $H$ in $G$ is 
\[
\left(1\pm O\left(\varepsilon\right)\right)d\left(G\right)^{e\left(H\right)-e\left(F\right)}n^{v\left(H\right)-k}.
\]
In particular, taking $k=0$, the number of copies of $H$ in $G$
is 
\[
\left(1\pm O\left(\varepsilon\right)\right)d\left(G\right)^{e\left(H\right)}\frac{n^{v\left(H\right)}}{\left|\operatorname{Aut}\left(H\right)\right|},
\]
where $\left|\operatorname{Aut}\left(H\right)\right|$ is the number
of automorphisms of $H$.
\end{prop}

\begin{proof}
Let $U=\left\{ u_{1},\dots,u_{k}\right\} $; we proceed by induction
on the number of vertices in $V\left(H\right)\backslash U$. The base
case is where $U=V\left(H\right)$, which is trivial. Suppose there
is a vertex $v\in V\left(H\right)\backslash U$; by induction there
are
\[
\left(1\pm O\left(\varepsilon\right)\right)d\left(G\right)^{e\left(H\right)-e\left(F\right)-\deg_{H}\left(v\right)}n^{v\left(H\right)-\left|U\right|-1}
\]
embeddings of $H-v$ extending $\phi$. For each such embedding, by
$\left(\varepsilon,v\left(H\right)-1\right)$-quasirandomness, there
are $\left(1\pm\varepsilon\right)d\left(G\right)^{\deg_{H}\left(v\right)}n$
ways to choose a vertex of $G$ with the right adjacencies to complete
the embedding of $H$. The desired result follows.
\end{proof}
Now we are ready to prove \ref{lem:greedy-random-approx-uniform}.
\begin{proof}
Each $G\left(S_{i}\right)$ has 
\[
\left(1\pm O\left(n^{-a}\right)\right)\left(1-\frac{i}{N}\right)^{3}\frac{n^{3}}{6}
\]
triangles, by $\left(n^{-a},2\right)$-quasirandomness and \ref{prop:H-count}.
We therefore have
\[
\Pr\left(\randS=S\right)=\prod_{i=0}^{\alpha N-1}\frac{1}{\left(1\pm O\left(n^{-a}\right)\right)\left(1-i/N\right)^{3}n^{3}/6},
\]
and a similar expression holds for $\Pr\left(\randS=S'\right)$. Taking
quotients term-by-term gives
\begin{align*}
\frac{\Pr\left(\randS=S\right)}{\Pr\left(\randS=S'\right)} & \le\left(1+O\left(n^{-a}\right)\right)^{\alpha N}\\
 & \le\exp\left(O\left(n^{2-a}\right)\right)
\end{align*}
as desired.
\end{proof}

\subsection{A coupling lemma and a concentration inequality\label{sec:nibble-TRP-coupling}}

In this subsection we prove two lemmas that will be useful in combination
with \ref{lem:triangle-removal-transfer}. First, after some definitions
we will show how to couple the triangle removal process with a simpler
random hypergraph distribution.
\begin{defn}
For a partial system $S$, let $\Gnp Sp$ be the random distribution
on $3$-uniform hypergraphs where each hyperedge not conflicting with
$S$ (that is, not intersecting a hyperedge of $S$ in more than $2$
vertices) is included with probability $p$. So, if $\varnothing$
is the empty order-$n$ partial system, then $\Gnp{\varnothing}p=:\Gnp np$
is the standard binomial random 3-uniform hypergraph. Let $\GSnp Sp$
be the distribution on partial systems obtained from $\Gnp Sp$ by
considering all hyperedges which intersect another hyperedge in more
than $2$ vertices, and deleting all these hyperedges (and let $\GSnp np=\GSnp{\varnothing}p$).
Let $\RR\left(S,m\right)$ be the partial system distribution obtained
with $m$ steps of the triangle removal process starting from $G\left(S\right)$.
\end{defn}

Note that ${n \choose 3}/n=\left(1+o\left(1\right)\right)N$. For
small $\alpha>0$, we can view $\GSnp S{\alpha/n}$ as a ``bite''
of a ``nibbling'' process, that should be comparable to $\rg S{\alpha N}$.
\begin{lem}
\label{lem:bite-transfer}Let $\mathcal{P}$ be a property of unordered
partial systems that is monotone increasing in the sense that $S\in\mathcal{P}$
and $S'\supseteq S$ implies $S'\in\mathcal{P}$. Fix $\alpha\in\left(0,1\right)$
and $S\in\ordm m$ for some $m\le N-\alpha N$. Let $\randS\sim\rg S{\alpha N}$
and $\randS^{*}\sim\GSnp S{\alpha/n}$. Then
\[
\Pr\left(S\cup\randS\notin\mathcal{P}\right)=O\left(1\right)\Pr\left(S\cup\randS^{*}\notin\mathcal{P}\right),
\]
where $S\cup S'$ denotes the unordered partial system containing
all the edges of $S$ and of $S'$.
\end{lem}

Before proving \ref{lem:bite-transfer}, we remark that we abuse notation
slightly and use the conventions that $*$ is a superset of every
partial system, that $*\in\mathcal{P}$, and that $S\cup*=*$. That
is, if the triangle removal process aborts before reaching $\alpha N$
edges, we still say that it satisfies $\mathcal{P}$.
\begin{proof}
Let $\randS^{*}\sim\GSnp S{\alpha/n}$ be obtained from $\randG\sim\Gnp S{\alpha/n}$.
If $e\left(\randG\right)\le\alpha N$, then $\randS^{*}$ can be coupled
as a subset of $\randS$. Indeed, a random ordering of the edges of
$\randG$ can be viewed as the first few elements of a random ordering
of the set of triangles of $G\left(S\right)$, and the triangle removal
process with this ordering produces a superset of $\randS^{*}$. It
follows that
\[
\Pr\left(S\cup\randS\notin\mathcal{P}\right)\le\Pr\left(S\cup\randS^{*}\notin\mathcal{P}\cond e\left(\randG\right)\le\alpha N\right).
\]
Next, note that $e\left(\randG\right)$ has a binomial distribution
with mean $\E e\left(\randG\right)={n \choose 3}\alpha/n\le\alpha N$,
so it is easy to see that $\Pr\left(e\left(\randG\right)\le\alpha N\right)=\Omega\left(1\right)$.
It follows that
\begin{align*}
\Pr\left(S\cup\randS\notin\mathcal{P}\right) & \le\Pr\left(S\cup\randS^{*}\notin\mathcal{P}\right)/\Pr\left(e\left(\randG\right)\le\alpha N\right)=O\left(1\right)\Pr\left(S\cup\randS^{*}\notin\mathcal{P}\right).\tag*{\qedhere}
\end{align*}
\end{proof}
We remark that \ref{lem:bite-transfer} is similar in spirit to \cite[Lemma~6.1]{War14}.
In this subsection we also state and prove a bounded-differences inequality
with Bernstein-type tails which can be used to analyse $\GSnp S{\alpha/n}$.
Standard bounded-difference inequalities such as the Azuma\textendash Hoeffding
inequality do not provide strong enough tail bounds to apply \ref{lem:triangle-removal-transfer}.
\begin{thm}
\label{thm:bernstein-type}Let $\boldsymbol{\omega}=\left(\boldsymbol{\omega}_{1},\dots,\boldsymbol{\omega}_{n}\right)$
be a sequence of independent, identically distributed random variables
with $\Pr\left(\boldsymbol{\omega}_{i}=1\right)=p$ and $\Pr\left(\boldsymbol{\omega}_{i}=0\right)=1-p$.
Let $f:\left\{ 0,1\right\} ^{n}\to\RR$ satisfy the Lipschitz condition
$\left|f\left(\boldsymbol{\omega}\right)-f\left(\boldsymbol{\omega}'\right)\right|\le K$
for all pairs $\boldsymbol{\omega},\boldsymbol{\omega}'\in\left\{ 0,1\right\} ^{n}$
differing in exactly one coordinate. Then
\[
\Pr\left(\left|f\left(\boldsymbol{\omega}\right)-\E f\left(\boldsymbol{\omega}\right)\right|>t\right)\le\exp\left(-\frac{t^{2}}{4K^{2}np+2Kt}\right).
\]
\end{thm}

\begin{proof}
We use Freedman's inequality (\ref{lem:freedman}), with the Doob
martingale $\randX\left(0\right),\dots,\randX\left(n\right)$ defined
by $\boldsymbol{X}\left(i\right)=\E\left[f\left(\boldsymbol{\omega}\right)\cond\boldsymbol{\omega}_{1},\dots,\boldsymbol{\omega}_{i}\right]$.
With $\Delta\randX\left(i\right)$ as the one-step change $\randX\left(i+1\right)-\randX\left(i\right)$
and with $V\left(i\right)=\sum_{i=0}^{i}\E\left[\left(\Delta\randX\left(i\right)\right)^{2}\cond\boldsymbol{\omega}_{1},\dots,\boldsymbol{\omega}_{j}\right]$,
it suffices to show that $V\left(n\right)\le2K^{2}np$ with probability
1.

Condition on $\boldsymbol{\omega}_{1},\dots,\boldsymbol{\omega}_{i}$
(thereby conditioning on $\randX\left(i\right)$). Let $X^{0}$ and
$X^{1}$ be the values of $\randX\left(i+1\right)$ in the cases $\boldsymbol{\omega}_{i+1}=0$
and $\boldsymbol{\omega}_{i+1}=1$, respectively. We have 
\begin{align*}
\randX\left(i\right) & =pX^{1}+\left(1-p\right)X^{0},\\
\left|\randX\left(i\right)-X^{0}\right| & =p\left|X^{1}-X^{0}\right|\le Kp.
\end{align*}
So, 
\begin{align*}
\E\left[\left(\Delta\randX\left(i\right)\right)^{2}\cond\boldsymbol{\omega}_{1},\dots,\boldsymbol{\omega}_{i}\right] & =p\left(\randX\left(i\right)-X^{1}\right)^{2}+\left(1-p\right)\left(\randX\left(i\right)-X^{0}\right)^{2}\\
 & \le K^{2}p+\left(1-p\right)K^{2}p^{2}\\
 & \le2K^{2}p.
\end{align*}
The desired bound on $V\left(n\right)$ follows.
\end{proof}
Since the first version of this paper, we learned that \ref{thm:bernstein-type}
is also a direct consequence of \cite[Theorem~1.3]{War16}.

\section{\label{sec:completion-count}Counting completions of Steiner triple
systems}

\global\long\def\coG{G}%

\global\long\def\Sext#1{\mathcal{S}^{*}\left(#1\right)}%

\global\long\def\randl{\boldsymbol{\lambda}}%

\global\long\def\randz{\boldsymbol{z}}%

\global\long\def\randR{\boldsymbol{R}}%

In this section we prove \ref{lem:num-extensions}. This is accomplished
with minor adaptations of proofs by Linial and Luria \cite{LL13}
and Keevash \cite{Kee15}. As in \ref{sec:random-STS}, let $N={n \choose 2}/3$
and assume that $n$ is 1 or 3 mod 6.

For a partial system $S\in\S_{\alpha N}^{n^{-a}}$, let $\Sext S$
be the number of Steiner triple systems that include $S$. We want
to determine $\left|\ext S\right|=\left(N-\alpha N\right)!\left|\Sext S\right|$
up to a factor of $e^{n^{2-b}}$ (for some $b>0$).

First, we can get an upper bound via the \emph{entropy method}, as
used by Linial and Luria \cite{LL13}. The reader may wish to refer
to that paper for more detailed exposition.

Before we begin the proof, we briefly remind the reader of the basics
of the notion of entropy. For random elements $\randX,\randY$ with
supports $\supp\randX$, $\supp\randY$, we define the (base-$e$)
\emph{entropy }
\[
H\left(\randX\right)=-\sum_{x\in\supp\randX}\Pr\left(\randX=x\right)\log\left(\Pr\left(\randX=x\right)\right)
\]
and the\emph{ conditional entropy}
\[
H\left(\randX\cond\randY\right)=\sum_{y\in\supp\randY}\Pr\left(\randY=y\right)H\left(\randX\cond\randY=y\right).
\]
We will use two basic properties of entropy. First, we always have
$H\left(\randX\right)\le\log\,\left|\supp\randX\right|$, with equality
only when $\randX$ has the uniform distribution on its support. Second,
for any sequence of random elements $\randX_{1},\dots,\randX_{n}$,
we have
\[
H\left(\randX_{1},\dots,\randX_{n}\right)=\sum_{i=1}^{n}H\left(\randX_{i}\cond\randX_{1},\dots,\randX_{i-1}\right).
\]
See for example \cite{CT12} for an introduction to the notion of
entropy and proofs of the above two facts.
\begin{thm}
\label{thm:number-of-completions-upper}For any $a>0$, any $\alpha\in\left[0,1\right]$,
and any $S\in\cSm{n^{-a},2}{\alpha N}$,
\[
\left|\Sext S\right|\le\left(\left(1+O\left(n^{-a}+n^{-1/2}\right)\right)\left(\frac{1-\alpha}{e}\right)^{2}n\right)^{N\left(1-\alpha\right)}.
\]
\end{thm}

\begin{proof}
Let $\randS^{*}\in\Sext S$ be a uniformly random completion of $S$.
We will estimate the entropy $H\left(\randS^{*}\right)=\log\,\left|\Sext S\right|$
of $\randS^{*}$.

Let $\coG=G\left(S\right)$. For each $e=\left\{ x,y\right\} \in\coG$,
let $\left\{ x,y,\randz_{e}\right\} $ be the hyperedge that includes
$e$ in $\randS^{*}$. So, the sequence $\left(\randz_{e}\right)_{e\in\coG}$
determines $\randS^{*}$. For any ordering on the edges of $\coG$,
we have
\begin{equation}
H\left(\randS^{*}\right)=\sum_{e\in\coG}H\left(\randz_{e}\cond\left(\randz_{e'}\,\colon\,e'<e\right)\right).\label{eq:H(S)}
\end{equation}
Now, a sequence $\lambda\in\left[0,1\right]^{E\left(\coG\right)}$
with all $\lambda_{e}$ distinct induces an ordering on the edges
of $\coG$, with $e'<e$ when $\lambda_{e'}>\lambda_{e}$. Let $\randR_{e}\left(\lambda\right)$
be an upper bound on $\left|\supp\left(\randz_{e}\cond\randz_{e'}\,\colon\,\lambda_{e'}>\lambda_{e}\right)\right|$
defined as follows. $\randR_{e}\left(\lambda\right)=1$ if $\lambda_{\left\{ x,\randz_{e}\right\} }>\lambda_{e}$
or $\lambda_{\left\{ y,\randz_{e}\right\} }>\lambda_{e}$ (because
in this case $\randz_{e}$ is determined). Otherwise, $\randR_{e}\left(\lambda\right)$
is 1 plus the number of vertices $v\notin\left\{ x,y,\randz_{e}\right\} $
such that $\left\{ x,v\right\} ,\left\{ y,v\right\} \in\coG$, and
$\lambda_{e'}<\lambda_{e}$ for each of the 6 edges $e'\in\coG$ included
in the hyperedges that include $\left\{ x,v\right\} $ and $\left\{ y,v\right\} $
in $\randS^{*}$. Since $\randR_{e}\left(\lambda\right)$ is an upper
bound on $\left|\supp\left(\randz_{e}\cond\randz_{e'}\,\colon\,\lambda_{e'}>\lambda_{e}\right)\right|$,
we have
\begin{equation}
H\left(\randz_{e}\cond\randz_{e'}\,\colon\,\lambda_{e'}>\lambda_{e}\right)\le\E\left[\log\randR_{e}\left(\lambda\right)\right].\label{eq:H(z)}
\end{equation}
It follows from \ref{eq:H(S)} and \ref{eq:H(z)} that
\[
H\left(\randS^{*}\right)\le\sum_{e\in\coG}\E\left[\log\randR_{e}\left(\lambda\right)\right].
\]
This is true for any fixed $\lambda$, so it is also true if $\lambda$
is chosen randomly, as follows. Let $\randl=\left(\randl_{e}\right)_{e\in\coG}$
be a sequence of independent random variables, where each $\randl_{e}$
has the uniform distribution in $\left[0,1\right]$. (With probability
1 each $\randl_{v}$ is distinct). Then
\[
H\left(\randS^{*}\right)\le\sum_{e\in\coG}\E\left[\log\randR_{e}\left(\randl\right)\right].
\]
Next, for any $S^{*}\in\Sext S$ and $\lambda_{e}\in\left[0,1\right]$,
let 
\[
R_{e}^{S^{*},\lambda_{e}}=\E\left[\randR_{e}\left(\randl\right)\cond\randS^{*}=S^{*},\,\randl_{e}=\lambda_{e},\,\randl_{\left\{ x,\randz_{e}\right\} },\randl_{\left\{ y,\randz_{e}\right\} }<\randl_{e}\right].
\]

(Note that $\randl_{e}=\lambda_{e}$ occurs with probability zero,
so formally we should condition on $\randl_{e}=\lambda_{e}\pm\d\lambda_{e}$
and take limits in what follows, but there are no continuity issues
so we will ignore this detail). Now, in $\coG$, by $\left(n^{-a},2\right)$-quasirandomness
$x$ and $y$ have $\left(1+O\left(n^{-a}\right)\right)\left(1-\alpha\right)^{2}n$
common neighbours other than $\randz_{e}$. By the definition of $\randR_{e}\left(\randl\right)$
and linearity of expectation, we have 
\[
R_{e}^{S^{*},\lambda_{e}}=1+\left(1+O\left(n^{-a}\right)\right)\left(1-\alpha\right)^{2}\lambda_{e}^{6}n.
\]
By Jensen's inequality,
\[
\E\left[\log\randR_{e}\left(\randl\right)\cond\randS^{*}=S^{*},\,\randl_{e}=\lambda_{e},\,\randl_{\left\{ x,\randz_{e}\right\} },\randl_{\left\{ y,\randz_{e}\right\} }<\randl_{e}\right]\le\log R_{e}^{S^{*},\lambda_{e}},
\]
and 
\[
\Pr\left(\randl_{\left\{ x,\randz_{e}\right\} },\randl_{\left\{ y,\randz_{e}\right\} }<\randl_{e}\cond\randl_{e}=\lambda_{e}\right)=\lambda_{e}^{2},
\]
so
\[
\E\left[\log\randR_{e}\left(\randl\right)\cond\randS^{*}=S^{*},\,\randl_{e}=\lambda_{e}\right]\le\lambda_{e}^{2}\log R_{e}^{S^{*},\lambda_{e}}+\left(1-\lambda_{e}^{2}\right)\log1=\lambda_{e}^{2}\log R_{e}^{S^{*},\lambda_{e}}.
\]
We then have 
\begin{align*}
\E\left[\log\randR_{e}\left(\randl\right)\cond\randS^{*}=S^{*}\right] & \le\E\left[\lambda_{e}^{2}\log R_{e}^{S^{*},\randl_{e}}\right]\\
 & =\int_{0}^{1}\lambda_{e}^{2}\log\left(1+\left(1+O\left(n^{-a}\right)\right)\left(1-\alpha\right)^{2}\lambda_{e}^{6}n\right)\d\lambda_{e}.
\end{align*}
For $C>0$ we can compute
\begin{align}
\int_{0}^{1}t^{2}\log\left(1+Ct^{6}\right)\d t & =\frac{1}{3}\left(\log\left(1+C\right)-2+\frac{2\arctan\sqrt{C}}{\sqrt{C}}\right),\label{eq:complicated-integral}
\end{align}
so (taking $C=\left(1+O\left(n^{-a}\right)\right)\left(1-\alpha\right)^{2}n$)
we deduce
\[
\E\left[\log\randR_{e}\left(\randl\right)\cond\randS^{*}=S^{*}\right]\le\frac{1}{3}\left(\log\left(\left(1-\alpha\right)^{2}n\right)-2\right)+O\left(n^{-a}+n^{-1/2}\right).
\]
We conclude that
\begin{align*}
\log\,\left|\Sext S\right| & \le H\left(\randS^{*}\right)\\
 & \le\sum_{e\in\coG}\E\left[\log\randR_{e}\left(\randl\right)\right]\\
 & \le\left(N-\alpha N\right)\left(\log\left(\left(1-\alpha\right)^{2}n\right)-2+O\left(n^{-a}+n^{-1/2}\right)\right),
\end{align*}
which is equivalent to the theorem statement.
\end{proof}
For the lower bound, we will count ordered Steiner triple systems.
\begin{thm}
\label{thm:number-of-completions-lower}Fixing sufficiently large
$h\in\NN$ and any $a>0$, there is $b=b\left(a,h\right)>0$ such
that the following holds. For any $\alpha\in\left(0,1\right)$ and
any $S\in\cordm{n^{-a},h}{\alpha N}$,
\[
\left|\ext S\right|\ge\left(\left(1-O\left(n^{-b}\right)\right)\left(\frac{1-\alpha}{e}\right)^{2}n\right)^{N\left(1-\alpha\right)}\left(N-\alpha N\right)!.
\]
\end{thm}

To prove \ref{thm:number-of-completions-lower} we will need an analysis
of the triangle removal process (which we provide in \ref{sec:random-triangle-removal})
and the following immediate consequence of \cite[Theorem~2.1]{Kee15}.
\begin{thm}
\label{thm:keevash}There are $h\in\NN$, $\varepsilon_{0},a\in\left(0,1\right)$
and $n_{0},\ell\in\NN$ such that if $S\in\cSm{\varepsilon,h}m$ is
a partial system with $n\ge n_{0}$, $d\left(G\left(S\right)\right)=1-m/N\ge n^{-a}$
and $\varepsilon\le\varepsilon_{0}d\left(G\right)^{\ell}$, then $S$
can be completed to a Steiner triple system.
\end{thm}

\global\long\def\extm#1#2{\mathcal{O}_{#1}^{\mathrm{ext}}\left(#2\right)}%

\begin{proof}[Proof of \ref{thm:number-of-completions-lower}]
Let $h\ge2$, $\ell$, $\varepsilon_{0}$ be as in \ref{thm:keevash}.
Let $c>0$ be smaller than $a\cdot b\left(a,h\right)$ in the notation
of \ref{lem:triangle-removal-analysis}, and smaller than the ``$a$''
in \ref{thm:keevash}. Let $\varepsilon=n^{-c/\ell}/\varepsilon_{0}$
and $M=\left(1-\varepsilon\right)N$. Let $\randS^{*}\in\ordm M\cup\left\{ *\right\} $
be the result of running the triangle removal process on $G\left(S\right)$
to build a partial system extending $S$, until there are $M$ hyperedges
(that is, $\randS^{*}=S\cup\randS$ where $\randS\sim\RR\left(S,M-\alpha N\right)$
in the notation of \ref{sec:nibble-TRP-coupling}). Let $\ord^{*}$
be the set of $M$-hyperedge $\left(\varepsilon,h\right)$-quasirandom
ordered partial systems $S^{*}\in\cordm{\varepsilon,h}M$ extending
$S$. The choice of $c$ ensures that by \ref{lem:triangle-removal-analysis}
we a.a.s. have $\randS^{*}\in\ord^{*}$, and by \ref{thm:keevash}
each $S^{*}\in\ord^{*}$ can be completed to an ordered Steiner triple
system.

Now, by \ref{prop:H-count} and quasirandomness, for each $S^{*}\in\ord^{*}$,
the number of triangles in each $G\left(S_{i}^{*}\right)$ is 
\[
\left(1\pm O\left(n^{-c}\right)\right)\left(1-i/N\right)^{3}n^{3}/6,
\]
so
\[
\Pr\left(\randS^{*}=S^{*}\right)\le\prod_{i=\alpha N}^{M-1}\frac{1}{\left(1-O\left(n^{-c}\right)\right)\left(1-i/N\right)^{3}n^{3}/6}.
\]
As discussed, using \ref{lem:triangle-removal-analysis} we have
\[
\sum_{S^{*}\in\ord^{*}}\Pr\left(\randS^{*}=S^{*}\right)=1-o\left(1\right),
\]
so 
\begin{align*}
\left|\ord^{*}\right| & \ge\left(1-o\left(1\right)\right)\prod_{i=\alpha N}^{M-1}\left(1-O\left(n^{-c}\right)\right)\left(1-\frac{i}{N}\right)^{3}\frac{n^{3}}{6}\\
 & =\left(\left(1-O\left(n^{-c}\right)\right)\frac{n^{3}}{6}\right)^{\left(1-\alpha\right)N}\exp\left(3\sum_{i=\alpha N}^{M-1}\log\left(1-\frac{i}{N}\right)\right).
\end{align*}
Now, note that 
\[
\sum_{i=\alpha N}^{M-1}\frac{1}{N}\log\left(1-\frac{i+1}{N}\right)\le\int_{\alpha}^{\left(1-\varepsilon\right)}\log\left(1-t\right)\d t\le\sum_{i=\alpha N}^{M-1}\frac{1}{N}\log\left(1-\frac{i}{N}\right).
\]
We compute
\begin{align*}
\sum_{i=\alpha N}^{M}\left(\log\left(1-\frac{i}{N}\right)-\log\left(1-\frac{i+1}{N}\right)\right) & =\sum_{i=\alpha N}^{M}\log\left(1+\frac{1}{N-\left(i+1\right)}\right)\\
 & \le\sum_{i=\alpha N}^{M}\frac{1}{N-\left(i+1\right)}\\
 & =O\left(\log n\right),
\end{align*}
so, noting that $\int\log s\d s=s\left(\log s-1\right)$,
\begin{align*}
3\sum_{i=\alpha N}^{M}\log\left(1-\frac{i}{N}\right) & =3N\int_{\alpha}^{\left(1-\varepsilon\right)}\log\left(1-t\right)\d t+O\left(\log n\right)\\
 & =3N\int_{\varepsilon}^{\left(1-\alpha\right)}\log s\d s+O\left(\log n\right)\\
 & =3N\left(\left(1-\alpha\right)\left(\log\left(1-\alpha\right)-1\right)-\varepsilon\left(\log\varepsilon-1\right)\right)+O\left(\log n\right),\\
\exp\left(3\sum_{i=\alpha N}^{M}\log\left(1-\frac{i}{N}\right)\right) & =\left(\left(1+O\left(n^{-c/\ell}\log n\right)\right)\frac{1-\alpha}{e}\right)^{3N\left(1-\alpha\right)}.
\end{align*}
For $b<c/\ell$, it follows that
\begin{align*}
\left|\ord^{*}\right| & \ge\left(\left(1-O\left(n^{-b}\right)\right)\frac{n^{3}\left(1-\alpha\right)^{3}}{6e^{3}}\right)^{\left(1-\alpha\right)N}\\
 & =\left(\left(1-O\left(n^{-b}\right)\right)\left(\frac{1-\alpha}{e}\right)^{2}n\right)^{\left(1-\alpha\right)N}\left(N-\alpha N\right)!.
\end{align*}
Recalling that each $S^{*}\in\ord^{*}$ can be completed, the desired
result follows.
\end{proof}
Now, it is extremely straightforward to prove \ref{lem:num-extensions}.
\begin{proof}
Let $b\le\min\left\{ a,1/2\right\} $ and $h\ge2$ satisfy \ref{thm:number-of-completions-lower}.
By \ref{thm:number-of-completions-upper} we have 
\[
\left|\ext S\right|\le\left|\Sext S\right|\left(N-\alpha N\right)!\le\left(\left(1+O\left(n^{-b}\right)\right)\left(\frac{1-\alpha}{e}\right)^{2}n\right)^{N\left(1-\alpha\right)}\left(N-\alpha N\right)!,
\]
and by \ref{thm:number-of-completions-lower} we have
\[
\left|\ext{S'}\right|\ge\left(\left(1-O\left(n^{-b}\right)\right)\left(\frac{1-\alpha}{e}\right)^{2}n\right)^{N\left(1-\alpha\right)}\left(N-\alpha N\right)!.
\]
Dividing these bounds gives
\[
\frac{\left|\ext S\right|}{\left|\ext{S'}\right|}\le\left(1-O\left(n^{-b}\right)\right)^{N\left(1-\alpha\right)}\le\exp\left(O\left(n^{2-b}\right)\right).\tag*{\qedhere}
\]
\end{proof}

\section{\label{sec:random-triangle-removal}An analysis of the triangle removal
process}

\global\long\def\randQ{\boldsymbol{Q}}%

\global\long\def\randY{\boldsymbol{Y}}%

\global\long\def\randT{\boldsymbol{T}}%

\global\long\def\randP{\boldsymbol{P}}%

\global\long\def\randZ{\boldsymbol{Z}}%

The triangle removal process starting from the complete graph has
already been thoroughly analysed; see in particular the precise analysis
by Bohman, Frieze and Lubetzky  \cite{BFL15} and their simplified
analysis in \cite{BFL10}. In this paper, we will need an analysis
of the triangle removal process starting from a quasirandom graph.
This basically follows from the aforementioned work of Bohman, Frieze
and Lubetzky, but since the result we need was not stated in a concrete
form in their papers, we provide our own (very simplified, and very
crude) analysis in this section. We emphasise that this section contains
no new ideas.

As in \ref{sec:random-STS}, let $N={n \choose 2}/3$ and assume that
$n$ is 1 or 3 mod 6. As previously introduced, the triangle removal
process is defined as follows. We start with a graph $G$ with say
$3N-3m$ edges, then iteratively delete (the edges of) a triangle
chosen uniformly at random from all triangles in the remaining graph.
Let 
\[
G=\randG\left(m\right),\randG\left(m+1\right),\dots
\]
be the sequence of random graphs generated by this process. This process
cannot continue forever, but we ``freeze'' the process instead of
aborting it: if $\randG\left(\randM\right)$ is the first graph in
the sequence with no triangles, then let $\randG\left(i\right)=\randG\left(\randM\right)$
for $i\ge\randM$.

Our objective in this section is to show that if $G$ is quasirandom
then the triangle removal process is likely to maintain quasirandomness
and unlikely to freeze until nearly all edges are gone.
\begin{thm}
\label{lem:triangle-removal-analysis}For all $h\ge2$ and $a>0$
there is $b\left(a,h\right)>0$ such that the following holds. Let
$n^{-a}\le\varepsilon<1/2$ and suppose $G$ is a $\left(\varepsilon,h\right)$-quasirandom
graph with $N-3m$ edges. Then a.a.s. $\randM\ge\left(1-\varepsilon^{b}\right)N$
and moreover for each $m\le i\le\left(1-\varepsilon^{b}\right)N$,
the graph $\randG\left(i\right)$ is $\left(\varepsilon^{b},h\right)$-quasirandom.
\end{thm}

Note that $K_{n}$ is $\left(O\left(1/n\right),h\right)$-quasirandom
for any fixed $h$, so in particular when we start the triangle removal
process from $G=K_{n}$ it typically runs almost to completion.

To prove \ref{lem:triangle-removal-analysis}, it will be convenient
to use Freedman's inequality \cite[Theorem~1.6]{Fre75}, as follows.
(This was originally stated for martingales, but it also holds for
supermartingales with the same proof). Here and in what follows, we
write $\Delta X\left(i\right)$ for the one-step change $X\left(i+1\right)-X\left(i\right)$
in a variable $X$.
\begin{lem}
\label{lem:freedman}Let $\randX\left(0\right),\randX\left(1\right),\dots$
be a supermartingale with respect to a filtration $\left(\mathcal{F}_{i}\right)$.
Suppose that $\left|\Delta\randX\left(i\right)\right|\le K$ for all
$i$, and let $V\left(i\right)=\sum_{j=0}^{i-1}\E\left[\left(\Delta\randX\left(j\right)\right)^{2}\cond\mathcal{F}_{j}\right]$.
Then for any $t,v>0$,
\[
\Pr\left(\randX\left(i\right)\ge\randX\left(0\right)+t\mbox{ and }V\left(i\right)\le v\mbox{ for some }i\right)\le\exp\left(-\frac{t^{2}}{2\left(v+Kt\right)}\right).
\]
\end{lem}

\begin{proof}[Proof of \ref{lem:triangle-removal-analysis}]
For a set $A$ of at most $h$ vertices, let $\randY_{A}\left(i\right)=\left|\bigcap_{w\in A}N_{\randG\left(i\right)}\left(w\right)\right|$.
Let $p\left(i\right)=\left(1-i/N\right)$ and let $p^{k}\left(i\right)=\left(1-i/N\right)^{k}$,
so that $p^{\left|A\right|}\left(i\right)n$ is the predicted trajectory
of each $\randY_{A}\left(i\right)$.

Fix some large $C$ and small $c$ to be determined. We will choose
$b<c/\left(C+1\right)$ so that $e\left(i\right):=p\left(i\right)^{-C}\varepsilon^{c}\le\varepsilon^{b}$
for $i\le N\left(1-\varepsilon^{b}\right)$. This means that if the
conditions
\begin{align*}
\randY_{A}\left(i\right) & \le p^{\left|A\right|}\left(i\right)n\left(1+e\left(i\right)\right),\\
\randY_{A}\left(i\right) & \ge p^{\left|A\right|}\left(i\right)n\left(1-e\left(i\right)\right),
\end{align*}
are satisfied for all $A$, then $\randG\left(i\right)$ is $\left(e\left(i\right),h\right)$-quasirandom
(therefore $\left(\varepsilon^{b},h\right)$-quasirandom).

Let $\randT'$ be the smallest index $i\ge m$ such that for some
$A$, the above equations are violated (let $\randT'=\infty$ if this
never happens). Let $\randT=\randT'\land N\left(1-\varepsilon^{b}\right)$.
Define the stopped processes 
\begin{align*}
\randY_{A}^{+}\left(i\right) & =\randY_{A}\left(i\land\randT\right)-p^{\left|A\right|}\left(i\land\randT\right)n\left(1+e\left(i\land\randT\right)\right),\\
\randY_{A}^{-}\left(i\right) & =-\randY_{A}\left(i\land\randT\right)+p^{\left|A\right|}\left(i\land\randT\right)n\left(1-e\left(i\land\randT\right)\right).
\end{align*}
We want to show that for each $A$ and each $s\in\left\{ +,-\right\} $,
the process $\randY_{A}^{s}=\left(\randY_{A}^{s}\left(i\right),\randY_{A}^{s}\left(i+1\right),\dots\right)$
is a supermartingale, and then we want to use \ref{lem:freedman}
and the union bound to show that a.a.s. each $\randY_{A}^{s}$ only
takes negative values.

To see that this suffices to prove \ref{lem:triangle-removal-analysis},
note that if $i<\randT$ then by \ref{prop:H-count} the number of
triangles in $\randG\left(i\right)$ is
\[
\randQ\left(i\right)=\left(1\pm O\left(e\left(i\right)\right)\right)\frac{p^{3}\left(i\right)n^{3}}{6}>0.
\]
This means $\randT\le\randM$, so the event that each $\randY_{A}^{s}$
only takes negative values contains the event that each $\randG\left(i\right)$
is non-frozen and sufficiently quasirandom for $i\le N\left(1-\varepsilon^{b}\right)$.

Let $\randR_{A}\left(i\right)=\bigcap_{w\in A}N_{\randG\left(i\right)}\left(w\right)$,
so that $\randY_{A}\left(i\right)=\left|\randR_{A}\left(i\right)\right|$.
Fix $A$, and consider $x\in\randR_{A}\left(i\right)$, for $i<\randT$.
The only way we can have $x\notin\randR_{A}\left(i+1\right)$ is if
we remove a triangle containing an edge $\left\{ x,w\right\} $ for
some $w\in A$. Now, for each $w\in A$, the number of triangles in
$\randG\left(i\right)$ containing the edge $\left\{ x,v\right\} $
is $\left(1\pm O\left(e\left(i\right)\right)\right)p^{2}\left(i\right)n$
by \ref{prop:H-count}. The number of triangles containing $x$ and
more than one vertex of $A$ is $O\left(1\right)$. So,
\begin{align*}
\Pr\left(x\notin\randR_{A}\left(i+1\right)\right) & =\frac{1}{\randQ\left(i\right)}\left(\sum_{w\in A}\left(1\pm O\left(e\left(i\right)\right)\right)p^{2}\left(i\right)n-O\left(1\right)\right)\\
 & =\left(1\pm O\left(e\left(i\right)\right)\right)\frac{\left|A\right|}{p\left(i\right)N}.
\end{align*}
For $i<\randT$ we have $\left|\randR_{A}\left(i\right)\right|=\left(1\pm e\left(i\right)\right)p^{\left|A\right|}\left(i\right)n$,
so by linearity of expectation
\begin{align*}
\E\left[\Delta\randY_{A}\left(i\right)\cond\randG\left(i\right)\right] & =-\left(1\pm O\left(e\left(i\right)\right)\right)\frac{\left|A\right|p^{\left|A\right|-1}\left(i\right)n}{N}\\
 & =-\frac{\left|A\right|p^{\left|A\right|-1}\left(i\right)n}{N}+O\left(\frac{e\left(i\right)p^{\left|A\right|-1}\left(i\right)}{n}\right).
\end{align*}
Note also that we have the bound $\Delta\randY_{A}\left(i\right)\le2=O\left(1\right)$
(with probability 1). Also, for fixed $k$, we have
\begin{align*}
\Delta p^{k}\left(i\right) & =\left(1-\frac{i+1}{N}\right)^{k}-\left(1-\frac{i}{N}\right)^{k}\\
 & =\left(1-\frac{i}{N}\right)^{k}\left(\left(\frac{N-i-1}{N-i}\right)^{k}-1\right)\\
 & =p^{k}\left(i\right)\left(\left(1-\frac{1}{N-i}\right)^{k}-1\right)\\
 & =p^{k}\left(i\right)\left(-\frac{k}{N-i}+O\left(\frac{1}{\left(N-i\right)^{2}}\right)\right)\\
 & =-\frac{kp^{k-1}\left(i\right)}{N}\left(1+O\left(\frac{p\left(i\right)}{n^{2}}\right)\right)\\
 & =-\frac{kp^{k-1}\left(i\right)}{N}+o\left(\frac{e\left(i\right)p^{k-1}\left(i\right)}{n^{2}}\right),
\end{align*}
and with $ep^{k}$ denoting the pointwise product $i\mapsto e\left(i\right)p^{k}\left(i\right)$,
we then have
\begin{align*}
\Delta\left(ep^{k}\right)\left(i\right) & =\varepsilon^{c}\Delta p^{k-C}\left(i\right)\\
 & =\varepsilon^{c}\Theta\left(\frac{\left(C-k\right)p^{k-C-1}\left(i\right)}{N}\right)\\
 & =\Theta\left(\frac{\left(C-k\right)e\left(i\right)p^{k-1}\left(i\right)}{n^{2}}\right).
\end{align*}
For large $C$ it follows that
\begin{align*}
\E\left[\Delta\randY_{A}^{+}\left(i\right)\cond\randG\left(i\right)\right] & =\E\left[\Delta\randY_{A}\left(i\right)\cond\randG\left(i\right)\right]-\Delta p^{\left|A\right|}\left(i\right)n-\Delta\left(ep^{\left|A\right|}\right)\left(i\right)n\le0,
\end{align*}
and similarly
\[
\E\left[\Delta\randY_{A}^{-}\left(i\right)\cond\randG\left(i\right)\right]\le0
\]
for $i<\randT$. (For $i\ge\randT$ we trivially have $\Delta\randY_{A}^{s}\left(i\right)=0$)
Since each $\randY_{A}^{s}$ is a Markov process, it follows that
each is a supermartingale. Now, we need to bound $\Delta\randY_{A}^{s}\left(i\right)$
and $\E\left[\left(\Delta\randY_{A}^{s}\left(i\right)\right)^{2}\cond\randG\left(i\right)\right]$,
which is easy given the preceding calculations. First, recalling that
$\Delta\randY_{A}\left(i\right)=O\left(1\right)$ and noting that
$\Delta p^{k}\left(i\right),\Delta\left(ep^{k}\right)\left(i\right)=O\left(1/N\right)$
we immediately have $\left|\Delta\randY_{A}^{s}\left(i\right)\right|=O\left(1\right)$.
Noting in addition that $\E\left[\Delta\randY_{A}\left(i\right)\cond\randG\left(i\right)\right]=O\left(1/n\right)$,
we have
\begin{align*}
\E\left[\left(\Delta\randY_{A}^{s}\left(i\right)\right)^{2}\cond\randG\left(i\right)\right] & =O\left(\E\left[\Delta\randY_{A}^{s}\left(i\right)\cond\randG\left(i\right)\right]\right)=O\left(\frac{1}{n}\right).
\end{align*}
Since $\randT\le N$, we also have
\[
\sum_{i=0}^{\infty}\E\left[\left(\Delta\randY_{A}^{s}\left(i\right)\right)^{2}\cond\randG\left(i\right)\right]=O\left(\frac{N}{n}\right)=O\left(n\right).
\]
Provided $c<1$ (and recalling that $\varepsilon<1/2$), applying
\ref{lem:freedman} with $t=e\left(m\right)p^{\left|A\right|}\left(m\right)n-\varepsilon p^{\left|A\right|}\left(m\right)n=\Omega\left(n\varepsilon^{c}\right)$
and $v=O\left(n\right)$ then gives
\[
\Pr\left(\randY_{A}^{s}\left(i\right)>0\mbox{ for some }i\right)\le\exp\left(-O\left(n\varepsilon^{2c}\right)\right).
\]
So, if $2c<\log_{\varepsilon}n\le a$, the union bound over all $A,s$
finishes the proof.
\end{proof}

\section{Perfect matchings via absorbers\label{sec:absorbers}}

In this section we prove \ref{thm:PM-in-STS} using \ref{lem:triangle-removal-transfer}
and the absorbing method. First, we define our absorbers, which are
small rooted hypergraphs that can contribute to a perfect matching
in two different ways. We are careful to make the definition in such
a way that absorbers can be shown to appear in $\rg n{\alpha N}$
with extremely high probability.
\begin{defn}
\label{def:absorber}An \emph{absorber} for an ordered triple $\left(x,y,z\right)$
is a set of hyperedges of the form
\[
\left\{ \left\{ x,x_{1},x_{2}\right\} ,\left\{ y,y_{1},y_{2}\right\} ,\left\{ z,z_{1},z_{2}\right\} ,\left\{ w_{x},w_{y},w_{z}\right\} ,\left\{ x_{1},y_{2},w_{z}\right\} ,\left\{ y_{1},z_{2},w_{x}\right\} ,\left\{ z_{1},x_{2},w_{y}\right\} \right\} .
\]
We call $x,y,z$ the \emph{rooted vertices }and we call the other
nine vertices the \emph{external vertices}. Also, we say the three
hyperedges containing rooted vertices are \emph{rooted hyperedges},
and the other four hyperedges containing only external vertices are
\emph{external hyperedges}. Note that an absorber has a perfect matching
on its full set of 12 vertices (we call this the \emph{covering }matching),
and it also has a perfect matching on its external vertices (we call
this the \emph{non-covering }matching). See \ref{fig:absorber}.
\end{defn}

\begin{figure}[h]
\begin{center}
\includegraphics{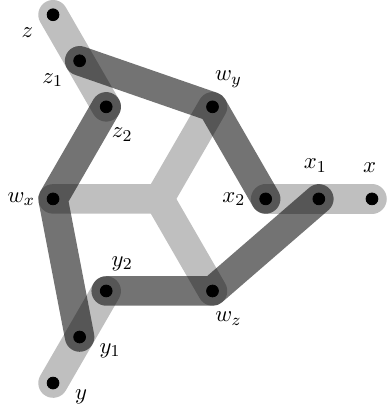}
\end{center}

\caption{\label{fig:absorber}An illustration of an absorber for $\left(x,y,z\right)$.
The light hyperedges are the covering matching and the dark hyperedges
are the non-covering matching.}
\end{figure}

Absorbers are the basic building blocks for a larger structure which
will eventually allow us to complete an almost-perfect matching into
a perfect matching. The relative positions of the absorbers in this
structure will be determined by a ``template'' with a ``resilient
matching'' property, as follows.
\begin{lem}
\label{lem:resilient-template}For any sufficiently large $n$, there
exists a 3-uniform hypergraph $T$ with $10n$ vertices, $120n$ hyperedges
and a set $Z$ of $2n$ vertices, such that if we remove any $n$
vertices from $Z$, the resulting hypergraph has a perfect matching.
We call $T$ a \emph{resilient template} and we call $Z$ its \emph{flexible
set.}
\end{lem}

To prove \ref{lem:resilient-template} we use the following lemma
of Montgomery \cite[Lemma~2.8]{Mon14}
\begin{lem}
\label{lem:montgomery}For any sufficiently large $n$, there exists
a bipartite graph $R$ with vertex parts $X$ and $Y\sqcup Z$, with
$\left|X\right|=3n$, $\left|Y\right|=\left|Z\right|=2n$, and maximum
degree 40, such that if we remove any $n$ vertices from $Z$, the
resulting bipartite graph has a perfect matching.
\end{lem}

\begin{proof}[Proof of \ref{lem:resilient-template}]
Consider the bipartite graph $R$ from \ref{lem:montgomery} on the
vertex set $X\sqcup\left(Y\sqcup Z\right)$. Note that $R$ has at
most $40\left|X\right|=120n$ edges (we can assume it has exactly
$120n$ edges, because adding edges does not affect the resilient
matching property). Add a set $W$ of $\left|X\right|$ new vertices
and put a perfect matching between $W$ and $X$, to obtain a $10n$-vertex
tripartite graph $R'$. Now, define a hypergraph $T$ on the same
vertex set by putting a hyperedge for each 3-vertex path running through
all three parts of $R'$ (we call such paths \emph{special paths}).
Note that an edge in $R$ can be uniquely extended to a special path
in $R'$, so $T$ has $120n$ hyperedges. Moreover, a matching in
$R$ can always be extended to a vertex-disjoint union of special
paths in $R'$, so $T$ is a resilient template with flexible set
$Z$.
\end{proof}
Now we can describe our absorbing structure in its entirety.
\begin{defn}
An \emph{absorbing structure} is a 3-uniform hypergraph $H$ of the
following form. Consider a resilient template $T$ and put externally
vertex-disjoint absorbers on each hyperedge of $T$, introducing 9
new vertices for each. Then delete the edges of $T$. That is, the
template just describes the relative positions of the absorbers, its
hyperedges are not actually in the absorbing structure.
\end{defn}

Note that an absorbing structure with a flexible set of size $2n$
has $10n+9\times120n=O\left(n\right)$ vertices and $7\times120n=O\left(n\right)$
hyperedges. An absorbing structure $H$ has the same crucial property
as the resilient template $T$ that defines it: if we remove half
of the vertices of the flexible set then what remains of $H$ has
a perfect matching. Indeed, after this removal we can find a perfect
matching $M$ of $T$, then our perfect matching of $H$ can be comprised
of the covering matching of the absorber on each hyperedge of $M$
and the non-covering matching for the absorber on each other hyperedge
of $T$. The existence of an absorbing structure, in addition to some
very weak pseudorandomness conditions, allows us to find perfect matchings
in a Steiner triple system, as follows.
\begin{lem}
\label{lem:PM-via-absorbers}Consider a Steiner triple system $S$
with vertex set $V$, $\left|V\right|=n\equiv3\mod6$, satisfying
the following conditions for some $\delta=\delta\left(n\right)=o\left(1/\log n\right)$
and fixed $\beta>0$.
\begin{enumerate}
\item \label{cond:absorber}There is an absorbing structure $H$ in $S$
with a flexible set $Z$ of size $6\floor{\delta^{2}n}$.
\item \label{cond:degree}For at most $\delta n$ of the vertices $v\in V\backslash V\left(H\right)$,
we have $\left|\left\{ \left\{ x,y\right\} \subseteq Z:\left\{ v,x,y\right\} \in E\left(S\right)\right\} \right|<6\delta^{5}n$.
That is to say, few vertices have unusually low degree into the flexible
set $Z$, in $S$.
\item \label{cond:density}Every vertex subset $W\subseteq V$ with $\left|W\right|\ge3\delta^{5}n$
induces at least $\left(1-\beta\right)\left|W\right|^{3}/\left(6n\right)$
hyperedges.
\end{enumerate}
Then $S$ has
\begin{equation}
\left(\frac{n}{2e^{2}}\left(1-\beta-o\left(1\right)\right)\right)^{n/3}\label{eq:count-pm}
\end{equation}
perfect matchings.
\end{lem}

\begin{proof}
We first describe a procedure to build a perfect matching (provided
$n$ is sufficiently large), then we count the number of ways to perform
this procedure.

Let $U\subseteq V\backslash V\left(H\right)$ be the set of vertices
with unusually low degree into $Z$, as per \ref{cond:degree}. The
first few hyperedges of our matching will cover $U$, and will not
use any vertices of $H$. We can in fact choose these hyperedges one-by-one
in a greedy fashion: considering each $v\in U$ in any order, note
that $v$ is in $\left(n-1\right)/2$ hyperedges of the Steiner triple
system $S$, and at most $3\left|U\right|+v\left(H\right)\le18\delta n+3\left(10+9\times120\right)\floor{\delta^{2}n}=o\left(n\right)$
of these hyperedges involve a vertex of $H$ or a vertex in the hyperedges
chosen so far.

Now, let $n'\ge n-v\left(H\right)-3\left|U\right|\ge n-19\delta n$
be the number of vertices in $V\backslash V\left(H\right)$ remaining
unmatched. Note that $v\left(H\right)=3\left(10+9\times120\right)\floor{\delta^{2}n}$
is divisible by 3, so $n'$ is divisible by 3 as well. We next use
\ref{cond:density} to repeatedly choose a hyperedge induced by the
remaining unmatched vertices in $V\backslash V\left(H\right)$ until
there are only $3\floor{\delta^{5}n}$ such vertices remaining unmatched.
(This means we are choosing $m=\left(n'-3\floor{\delta^{5}n}\right)/3\ge n/3-20\delta n$
hyperedges). We call this step the \emph{main step}.

Next, we greedily extend our matching to cover the remaining vertices
in $V\backslash V\left(H\right)$. Considering each uncovered $v\in V\backslash V\left(H\right)$
in any order, recall that $v\notin U$ so by \ref{cond:degree} there
are at least $6\delta^{5}n$ hyperedges of $S$ containing $v$ and
two vertices of $Z$. We can therefore choose such a hyperedge avoiding
the (fewer than $2\times3\floor{\delta^{5}n}$) vertices in $Z$ used
so far, to extend our matching. We have now covered all of $V\backslash V\left(H\right)$
and $6\floor{\delta^{5}n}$ vertices of $Z$; we can then repeatedly
apply \ref{cond:density} to the uncovered vertices in $Z$ to extend
our matching to cover half of $Z$. By the crucial property of an
absorbing structure, we can find a perfect matching on the remaining
vertices, completing our perfect matching of $S$.

Now we analyse the number of ways to perform the above procedure.
It actually suffices to count the number of ways to make the ordered
sequence of choices in the main step, which is at least
\begin{align}
\prod_{i=1}^{m}\left(1-\beta\right)\frac{\left(n'-3i\right)^{3}}{6n} & =\left(\frac{\left(1-\beta\right)\left(n'\right)^{3}}{6n}\right)^{m}\exp\left(\sum_{i=1}^{m}3\log\left(1-3\frac{i}{n'}\right)\right)\label{eq:num-matchings}\\
 & \ge\left(\left(1-\beta-3\times19\delta\right)\frac{n^{2}}{6}\right)^{n/3}n^{-20\delta n}\exp\left(n'\sum_{i=1}^{m}\frac{3}{n'}\log\left(1-3\frac{i}{n'}\right)\right).\nonumber 
\end{align}
Now, noting that $\int\log s\d s=s\left(\log s-1\right)$, we have
the Riemann sum approximation
\begin{align*}
\sum_{i=1}^{m}\frac{3}{n'}\log\left(1-3\frac{i}{n'}\right) & =\int_{0}^{m/n'}3\log\left(1-3t\right)\d t+o\left(1\right)\\
 & =\int_{0}^{1/3-o\left(1\right)}3\log\left(1-3t\right)\d t+o\left(1\right)\\
 & =\int_{o\left(1\right)}^{1}\log s\d s+o\left(1\right)\\
 & =-1+o\left(1\right).
\end{align*}
So, the expression in \ref{eq:num-matchings} is at least
\begin{align*}
\left(\left(1-\beta+O\left(\delta\right)\right)\frac{n^{2}}{6}\right)^{n/3}n^{O\left(\delta n\right)}e^{-n-o\left(n\right)} & =\left(\left(1-\beta+o\left(1\right)\right)\frac{n^{2}}{6e^{3}}\right)^{n/3}.
\end{align*}
(Recall that we are assuming $\delta=o\left(1/\log n\right)$). This
is a lower bound for the number of \emph{ordered} perfect matchings
in $S$. So, we divide by $\left(n/3\right)!=\left(\left(1+o\left(1\right)\right)n/3e\right)^{n/3}$
(using Stirling's approximation) to obtain \ref{eq:count-pm}.
\end{proof}

\subsection{Absorbing conditions in the triangle removal process\label{sec:absorbers-in-TRP}}

In this section we prove that the conditions in \ref{lem:PM-via-absorbers}
(for say $\delta=1/\log^{2}n$ and arbitrarily small $\beta>0$) hold
in a random Steiner triple system, proving \ref{thm:PM-in-STS}. We
do this using \ref{lem:triangle-removal-transfer}, showing that the
same properties hold with probability $1-\exp\left(-\tilde{\Omega}\left(n^{2}\right)\right)$
in the triangle removal process. (A tilde over asymptotic notation
indicates that polylogarithmic factors are being ignored).

Fix a large constant $h\in\NN$ (we will see later exactly how large
it should be), and fix small $\alpha>0$ (which we will assume is
small enough to satisfy certain inequalities later in the proof).
Fix a set $Z$ of $6\floor{\delta^{2}n}$ vertices (say $Z=\range{6\floor{\delta^{2}n}}$);
we will eventually find an absorbing structure with $Z$ as a flexible
set.

\subsubsection{High degree into the flexible set\label{subsec:degree}}

If \ref{cond:degree} is violated, there is a set $W$ of $\floor{\delta n}$
vertices outside $Z$ each with degree less than $6\delta^{5}n$ into
$Z$. There are then fewer than $6\delta^{6}n^{2}$ hyperedges with
one vertex in $W$ and two vertices in $Z$. We show that it is extremely
unlikely that there is a set $W$ with this property.

We will use \ref{lem:bite-transfer}, so let $\randS^{*}\sim\GSnp n{\alpha/n}$
be obtained from $\randG\sim\Gnp n{\alpha/n}$. Consider a set $W$
of $\floor{\delta n}$ vertices outside $Z$. Let $\randY$ be the
number of hyperedges of $\randS^{*}$ with one vertex in $W$ and
two vertices in $Z$. That is to say, $\randY$ is the number of such
hyperedges in $\randG$ that are \emph{isolated} in the sense that
they do not intersect any other hyperedge of $\randG$ in more than
one vertex. There are $\Theta\left(\left(\delta n\right)\left(\delta^{2}n\right)^{2}\right)=\Theta\left(\delta^{5}n^{3}\right)$
possible hyperedges and each is present and isolated with probability
$\left(\alpha/n\right)\left(1-\alpha/n\right)^{O\left(n\right)}=\Theta\left(n^{-1}\right)$,
so $\E\randY=\Theta\left(\delta^{5}n^{2}\right)$. Now, adding a hyperedge
to $\randG$ can increase $\randY$ by at most 1, and removing a hyperedge
can increase $\randY$ by at most 3 (by making three hyperedges isolated).
So, by \ref{thm:bernstein-type},
\[
\Pr\left(\randY\le2\delta^{6}n^{2}\right)\le\Pr\left(\left|\randY-\E\randY\right|\ge\Theta\left(\delta^{5}n^{2}\right)\right)\le\exp\left(-\Omega\left(\frac{\left(\delta^{5}n^{2}\right)^{2}}{3^{2}{n \choose 3}\alpha/n+3\delta^{5}n^{2}}\right)\right)=\exp\left(-\tilde{\Omega}\left(n^{2}\right)\right).
\]
Since there are no more than $2^{n}$ choices for $W$, we can use
the union bound, \ref{lem:bite-transfer} (with $S=\varnothing$)
and \ref{lem:triangle-removal-transfer} (with no conditioning; that
is, $\mathcal{Q}=\ord_{\alpha N}\cup\left\{ *\right\} $) to prove
that \ref{cond:degree} of \ref{lem:PM-via-absorbers} holds a.a.s.
in a random Steiner triple system.

\subsubsection{Density in subsets\label{subsec:density}}

Now we deal with \ref{cond:density}. It is not immediately clear
that one can consider just the first few hyperedges of a random Steiner
triple system as in \ref{subsec:degree}, but the key observation
is that in a random \emph{ordered} Steiner triple system, by symmetry
the first $\alpha N$ hyperedges have the same distribution as the
hyperedges corresponding to any other choice of $\alpha N$ indices.

With $\randS^{*}$ and $\randG$ as in \ref{subsec:degree}, consider
a set $W\subseteq V$ with $\left|W\right|\ge3\delta^{5}n$ and redefine
$\randY$ to be the number of hyperedges of $\randS^{*}$ included
in $W$ (which is the number of such hyperedges in $\randG$ that
are isolated). There are $\left(1+o\left(1\right)\right)\left|W\right|^{3}/6$
possible hyperedges, and each is present and isolated in $\randG$
with probability $\left(\alpha/n\right)\left(1-\alpha/n\right)^{O\left(n\right)}=\left(\alpha/n\right)\left(1-O\left(\alpha\right)\right)$.
Reasoning as in \ref{subsec:degree}, with probability $1-\exp\left(-\tilde{\Omega}\left(n^{2}\right)\right)$
we have $\randY\ge\alpha\left(1-O\left(\alpha\right)\right)\left|W\right|^{3}/\left(6n\right)$.
The union bound, \ref{lem:bite-transfer} and \ref{lem:triangle-removal-transfer}
prove that if $\randS$ is a random Steiner triple system, then a.a.s.
every appropriate subset $W$ induces at least $\alpha\left(1-O\left(\alpha\right)\right)\left|W\right|^{3}/\left(6n\right)$
hyperedges in $\randS_{\alpha N}$. By symmetry this property also
holds a.a.s. in $\randS_{k\alpha N}\backslash\randS_{\left(k-1\right)\alpha N}$
for each $k\le1/\alpha$. So, a.a.s. every $W$ induces a total of
$\left(1-O\left(\alpha\right)\right)\left|W\right|^{3}/\left(6n\right)$
hyperedges in $\randS$. For $\beta$ a large multiple of $\alpha$,
\ref{cond:density} of \ref{lem:PM-via-absorbers} is then satisfied.

\subsubsection{\label{subsec:absorbers}Absorbers}

Finally we show how to find an absorbing structure for \ref{cond:absorber},
which is much more involved. The first step is to show that there
are many absorbers rooted on every triple of vertices. We cannot hope
to do this by na\"ively analysing $\Gnp n{\alpha/n}$ and using \ref{lem:triangle-removal-transfer}
as in \ref{subsec:degree,subsec:density}, because the probability
that a vertex is isolated is already too large. Instead we must use
\ref{lem:triangle-removal-transfer} in its full generality, conditioning
on the quasirandomness of the first few steps of the triangle removal
process.

Let $a$ be small enough for \ref{lem:triangle-removal-transfer},
and let 
\[
\mathcal{Q}=\left\{ *\right\} \cup\left\{ S\in\ordm{2\alpha N}:S_{\alpha N}\in\cordm{n^{-a},h}{\alpha N}\right\} \supseteq\cordm{n^{-a},h}{2\alpha N}.
\]
Let $\randS\sim\rg n{2\alpha N}$, and condition on any $\randS_{\alpha N}=S\in\cordm{n^{-a},h}{\alpha N}$.
We will use \ref{lem:bite-transfer} to analyse $\randS\backslash\randS_{\alpha N}\sim\rg S{\alpha N}$
via $\GSnp S{\alpha/n}$. So, let $\randS^{*}\sim\GSnp S{\alpha/n}$
be obtained from $\randG\sim\Gnp S{\alpha/n}$.

By quasirandomness, every vertex has degree $\left(1\pm n^{-a}\right)\left(1-\alpha\right)n$
in $G\left(S\right)$, so every vertex is in $\left(1\pm n^{-a}\right)\alpha n/2=\Omega\left(\alpha n\right)$
hyperedges of $S$. Consider vertices $x,y,z$. Say an \emph{absorber-extension}
is a collection of four hyperedges which can be combined with three
hyperedges of $S$ incident to $x,y,z$, to form an absorber on $\left(x,y,z\right)$.
($S$ provides the rooted hyperedges of an absorber, and an absorber-extension
provides the external hyperedges). Let $\randY$ be the maximum size
of a hyperedge-disjoint collection of absorber-extensions in $\randS^{*}$;
equivalently, $\randY$ is the maximal size of a collection of disjoint
isolated absorber-extensions in $\randG$. This particular choice
of random variable is crucial, and allows us to use \ref{thm:bernstein-type}.
(The idea comes from a similar random variable used by Bollob\'as
\cite{Bol88}). Adding a hyperedge to $\randG$ can increase the size
of a maximal collection of hyperedge-disjoint absorber-extensions
by at most one, and removing a hyperedge can cause at most three hyperedge-disjoint
absorber-extensions to become isolated. So, changing the presence
of a hyperedge in $\randG$ can change $\randY$ by at most 3, as
in \ref{subsec:degree,subsec:density}.
\begin{claim}
\label{claim:expected-absorbers}If $h$ is large enough and $\alpha$
is small enough, then $\E\randY=\Omega\left(n^{2}\right)$.
\end{claim}

To prove \ref{claim:expected-absorbers} we'll need a simple lemma
about quasirandom graphs.
\begin{lem}
\label{lem:new-quasirandom}Suppose $G$ is an $n$-vertex $\left(o\left(1\right),2\right)$-quasirandom
graph with density $\gamma$ satisfying $\gamma,1-\gamma=\Omega\left(1\right)$.
Consider sets $X,Y,Z$ of disjoint ordered pairs of vertices of $G$,
such that $\left|X\right|,\left|Y\right|,\left|Z\right|=\Omega\left(n\right)$.
Then there are at least $\gamma^{3}\left|X\right|\left|Y\right|\left|Z\right|/8-o\left(n^{3}\right)$
choices of $\left(x_{1},x_{2}\right)\in X$, $\left(y_{1},y_{2}\right)\in Y$,
$\left(z_{1},z_{2}\right)\in Z$ such that $\left\{ x_{1},y_{2}\right\} $,
$\left\{ y_{1},z_{2}\right\} $ and $\left\{ z_{1},x_{2}\right\} $
are all edges of $G$.
\end{lem}

\begin{proof}
This quasirandomness condition implies quasirandomness in the Chung\textendash Graham\textendash Wilson
sense (see for example \cite[Theorem~9.3.2]{AS04}). In particular,
for any (not necessarily disjoint) vertex sets $A,B$, there are at
least $\gamma\left|A\right|\left|B\right|/2-o\left(n^{2}\right)$
edges between them, and all but $o\left(n\right)$ vertices of $A$
have at least $\gamma\left|B\right|/2-o\left(n\right)$ neighbours
in $B$. (To prove this latter fact, note that if for any $\varepsilon=\Omega\left(1\right)$
there were a set $A'\subseteq A$ of $\varepsilon n$ vertices that
each had $\gamma\left|B\right|-\varepsilon n$ neighbours in $B$,
then there would be at most $\gamma\left|A'\right|\left|B\right|/2-\varepsilon^{2}n^{2}$
edges between $A'$ and $B$, contradicting quasirandomness).

So, for $\left|X\right|-o\left(n\right)$ choices of $\left(x_{1},x_{2}\right)\in X$,
there are subsets $Y_{x_{1},x_{2}}\subseteq Y$, $Z_{x_{1},x_{2}}\subseteq Z$
with $\left|Y_{x_{1},x_{2}}\right|=\left(\gamma\left|Y\right|/2-o\left(n\right)\right)$
and $\left|Z_{x_{1},x_{2}}\right|=\left(\gamma\left|Z\right|/2-o\left(n\right)\right)$,
such that $\left\{ x_{1},y_{2}\right\} $ and $\left\{ z_{1},x_{2}\right\} $
are edges of $G$ for each $\left(y_{1},y_{2}\right)\in Y_{x_{1},x_{2}}$,
$\left(z_{1},z_{2}\right)\in Z_{x_{1},x_{2}}$. For each such choice
of $x_{1},x_{2}$, there are $\gamma\left|Y_{x_{1},x_{2}}\right|\left|Z_{x_{1},x_{2}}\right|/2-o\left(n^{2}\right)$
choices of $\left(y_{1},y_{2}\right)\in Y_{x_{1},x_{2}}$, $\left(z_{1},z_{2}\right)\in Z_{x_{1},x_{2}}$
such that $\left\{ y_{1},z_{2}\right\} $ is an edge of $G$. In summary,
there are at least
\[
\left(\left|X\right|-o\left(n\right)\right)\left(\gamma\left(\gamma\left|Y\right|/2-o\left(n\right)\right)\left(\gamma\left|Z\right|/2-o\left(n\right)\right)/2-o\left(n^{2}\right)\right)=\gamma^{3}\left|X\right|\left|Y\right|\left|Z\right|/8-o\left(n^{3}\right)
\]
suitable choices of $\left(x_{1},x_{2}\right)\in X,\left(y_{1},y_{2}\right)\in Y,\left(z_{1},z_{2}\right)\in Z$.
\end{proof}
Now we prove \ref{claim:expected-absorbers}.
\begin{proof}[Proof of \ref{claim:expected-absorbers}]
Let $\randX$ be the total number of isolated absorber-extensions
in $\randG$ and let $\randZ$ be the number of pairs of hyperedge-intersecting
absorber-extensions in $\randG$. We can obtain a collection of disjoint
isolated absorber-extensions by considering the collection of all
isolated absorber-extensions and deleting one from each intersecting
pair, so $\randY\ge\randX-\randZ$ and $\E\randY\ge\E\randX-\E\randZ$.
We first estimate $\E\randX$.

First we show that there are $\Theta\left(\alpha^{3}n^{6}\right)$
possible absorber-extensions not conflicting with $S$. To this end,
we first show that there are $\Theta\left(\left(\alpha n\right)^{3}\right)$
ways to choose three disjoint hyperedges $e_{x}=\left\{ x,x_{1},x_{2}\right\} ,e_{y}=\left\{ y,y_{1},y_{2}\right\} ,e_{z}=\left\{ z,z_{1},z_{2}\right\} \in E\left(S\right)$,
such that $\left\{ x_{1},y_{2}\right\} $, $\left\{ y_{1},z_{2}\right\} $
and $\left\{ z_{1},x_{2}\right\} $ are all edges of $G\left(S\right)$.
Indeed, let $X$ (respectively, $Y$ and $Z$) be the set of all pairs
of vertices $\left\{ x_{1},x_{2}\right\} $ (respectively, $\left\{ y_{1},y_{2}\right\} $
and $\left\{ z_{1},z_{2}\right\} $) such that $\left\{ x,x_{1},x_{2}\right\} \in E\left(S\right)$
(respectively, $\left\{ y,y_{1},y_{2}\right\} \in E\left(S\right)$
and $\left\{ z,z_{1},z_{2}\right\} \in E\left(S\right)$). Note that
$\left|X\right|,\left|Y\right|,\left|Z\right|=\Theta\left(\alpha n\right)$,
arbitrarily choose an ordering for each of the constructed pairs,
and apply \ref{lem:new-quasirandom} to show that there are $\Theta\left(\left(1-\alpha\right)^{3}\left(\alpha n\right)^{3}\right)=\Theta\left(\left(\alpha n\right)^{3}\right)$
ways to choose $e_{x},e_{y},e_{z}$. Then, the number of ways to choose
an absorber-extension compatible with $e_{x},e_{y},e_{z}$ is precisely
the number of copies in $G\left(S\right)$ of a certain graph $F$
rooted on the vertices of $e_{x},e_{y},e_{z}$. (Specifically, $F$
is the graph obtained by taking the external hyperedges of the hypergraph
in \ref{def:absorber} and replacing each hyperedge with a triangle
on its vertex set). Provided $h$ is large enough, by \ref{prop:H-count}
the number of suitable copies of $F$ is $\left(1-\alpha\right)^{O\left(1\right)}n^{3}=\Theta\left(n^{3}\right)$.
(Note that strictly speaking we are over-counting, because it is possible
that an absorber-extension can contribute to multiple different absorbers,
but this constant factor will not bother us).

The probability that each possible absorber-extension appears and
is isolated in $\randG$ is 
\[
\Theta\left(\left(\alpha/n\right)^{4}\left(1-\alpha/n\right)^{O\left(n\right)}\right)=\Theta\left(\alpha^{4}n^{-4}\right),
\]
so $\E\randX=\Theta\left(\alpha^{7}n^{2}\right)$. Now, we estimate
$\E\randZ$. It will be convenient to consider \emph{labelled} absorbers
and absorber-extensions; for the hypergraph in \ref{def:absorber}
denote its hyperedges (in the same order as in \ref{def:absorber})
by 
\[
e_{x},\,e_{y},\,e_{z},\,e_{*},\,e_{1},\,e_{2},\,e_{3}.
\]
There are several possibilities for a hyperedge-intersecting pair
of distinct absorber-extensions.
\begin{itemize}
\item Suppose they intersect in one hyperedge. Each such pair appears with
probability $O\left(\left(\alpha/n\right)^{7}\right)$.
\begin{itemize}
\item Suppose the intersecting hyperedge is $e_{*}$ for one of the absorber-extensions
(say the second). There are $O\left(\left(\alpha n\right)^{6}n^{3}\right)$
possibilities for such a pair of absorber-extensions, as follows.
Choose the first absorber-extension in one of $O\left(\left(\alpha n\right)^{3}n^{3}\right)$
ways, and choose one of its hyperedges which will intersect with the
second absorber-extension. Then, the second absorber-extension is
determined by its choices for $e_{x},e_{y},e_{z}$.
\item Suppose the intersecting hyperedge is say $e_{1}$ for the second
absorber-extension. There are $O\left(\left(\alpha n\right)^{4}n^{5}\right)$
possible such pairs, as follows. After choosing the first absorber-extension,
and choosing its hyperedge which will be intersecting, the choices
for $e_{x}$ and $e_{y}$ for the second absorber-extension are already
determined (if a suitable choice exists at all), because in $S$ each
pair of vertices is included in at most one hyperedge. One of the
vertices of $e_{*}$ is already also determined, so the second absorber-extension
is determined by a choice of $e_{z}$ and two vertices of $e_{*}$.
\end{itemize}
\item Suppose they intersect in two hyperedge. Each such pair appears with
probability $O\left(\left(\alpha/n\right)^{6}\right)$.
\begin{itemize}
\item Suppose the intersecting hyperedges are say $e_{*}$ and $e_{1}$
for the second absorber-extension. There are $O\left(\left(\alpha n\right)^{4}n^{3}\right)$
possibilities for such a pair: after choosing the first absorber-extension
and its hyperedges which will be intersecting, the second absorber-extension
is determined by its choice for $e_{z}$.
\item Suppose the intersecting hyperedges are say $e_{1}$ and $e_{2}$
for the second absorber-extension. There are $O\left(\left(\alpha n\right)^{3}n^{4}\right)$
possibilities for such a pair: after choosing the first absorber-extension
and its hyperedges which will be intersecting, the second absorber-extension
is determined by a single vertex for $e_{*}$.
\end{itemize}
\item Note that choosing three of $e_{1},e_{2},e_{3},e_{*}$ determines
the other, so there are only $O\left(\left(\alpha n\right)^{3}n^{3}\right)$
possibilities for a pair of absorber-extensions intersecting in three
hyperedges. Each such pair appears with probability $O\left(\left(\alpha/n\right)^{5}\right)$.
\end{itemize}
In summary (for small $\alpha$), we have 
\begin{align*}
\E\randZ & =O\left(\left(\left(\alpha n\right)^{6}n^{3}+\left(\alpha n\right)^{4}n^{5}\right)\left(\alpha/n\right)^{7}+\left(\left(\alpha n\right)^{4}n^{3}+\left(\alpha n\right)^{3}n^{4}\right)\left(\alpha/n\right)^{6}+\left(\alpha n\right)^{3}n^{3}\left(\alpha/n\right)^{5}\right)\\
 & =O\left(\alpha^{11}n^{2}\right).
\end{align*}
So, $\E\randY\ge\Theta\left(\alpha^{7}n^{2}\right)-O\left(\alpha^{11}n^{2}\right)=\Theta\left(\alpha^{7}n^{2}\right)$.
\end{proof}
As in \ref{subsec:degree,subsec:density}, \ref{thm:bernstein-type}
proves that $\randY=\Omega\left(n^{2}\right)$ with probability $1-\exp\left(-\Omega\left(n^{2}\right)\right)$.
Note that if there are $\Omega\left(n^{2}\right)$ hyperedge-disjoint
absorber-extensions then there must in fact be $\Omega\left(n\right)$
externally vertex-disjoint absorbers rooted on $x,y,z$. We can find
these greedily; each vertex is involved in only $O\left(n\right)$
hyperedges of $\randS$, so removing $O\left(1\right)$ vertices from
consideration results in at most $O\left(n\right)$ hyperedges being
removed from consideration. By the union bound, \ref{lem:bite-transfer}
and \ref{lem:triangle-removal-transfer} (with $\mathcal{Q}$ as defined
at the beginning of this subsection), it follows that in a random
Steiner triple system, a.a.s. every triple of vertices has $\Omega\left(n\right)$
externally vertex-disjoint absorbers.

If a Steiner triple system has this property, then we can very straightfowardly
greedily build an absorbing structure with flexible set $Z$, as follows.
Choose a resilient template $T$ on the vertex set of $\randS$, such
that the flexible set is $Z$. Consider each hyperedge $\left(x,y,z\right)$
of $T$ in any order, and greedily choose an absorber in $\randS$
rooted on $\left(x,y,z\right)$, each of whose external vertices is
disjoint to the template and all absorbers chosen so far. We have
proved that \ref{cond:absorber} of \ref{lem:PM-via-absorbers} holds
a.a.s. in a random Steiner triple system, completing the proof of
\ref{thm:PM-in-STS}.

\section{An upper bound for the number of perfect matchings\label{sec:max-num-matchings}}

\global\long\def\randl{\boldsymbol{\lambda}}%

\global\long\def\randz{\boldsymbol{z}}%

\global\long\def\randR{\boldsymbol{R}}%

\global\long\def\rande{\boldsymbol{e}}%

\global\long\def\randQ{\boldsymbol{Q}}%

In this section we prove \ref{thm:maximum-PMs}, with the entropy
method. Recall the definition of entropy and its basic properties
from \ref{sec:completion-count}.
\begin{proof}[Proof of \ref{thm:maximum-PMs}]
Let $\mathcal{M}$ be the set of perfect matchings in $S$. Consider
a uniformly random $\randM\in\mathcal{M}$, so that $H\left(\randM\right)=\log\,\left|\mathcal{M}\right|$
is the entropy of $\randM$. Let $\randM_{v}$ be the hyperedge of
$\randM$ containing the vertex $v$, so that the sequence $\left(\randM_{v}\right)_{v\in\range n}$
determines $\randM$. For any ordering on the vertices of $S$,
\begin{equation}
H\left(\randM\right)=\sum_{v\in V}H\left(\randM_{v}\cond\randM_{v'}:v'<v\right).\label{eq:entropy-1}
\end{equation}
Now, a sequence $\lambda\in\left[0,1\right]^{n}$ with all $\lambda_{v}$
distinct induces an ordering on $\range n$, with $v'<v$ when $\lambda_{v'}>\lambda_{v}$.
Let $\randR_{v}\left(\lambda\right)$ be 1 plus the number of hyperedges
$e\ne\randM_{v}$ containing $v$ in $S$ such that $\lambda_{v'}<\lambda_{v}$
for all $v'\in\left(\bigcup_{z\in e}\randM_{z}\right)\backslash\left\{ v\right\} $.
(In particular, $\randR_{v}\left(\lambda\right)=1$ if $\lambda_{v'}>\lambda_{v}$
for some $v'\in\randM_{v}\backslash\left\{ v\right\} $, in which
case $\randM_{v}$ is determined by the information $\left(\randM_{v'}\,\colon\,\lambda_{v'}>\lambda_{v}\right)$).
Note that $\randR_{v}\left(\lambda\right)$ is an upper bound on $\left|\supp\left(\randM_{v}\cond\randM_{v'}\,\colon\,\lambda_{v'}>\lambda_{v}\right)\right|$,
and therefore
\begin{equation}
H\left(\randM_{v}\cond\randM_{v'}:\lambda_{v'}>\lambda_{v}\right)\le\E\left[\log\randR_{v}\left(\lambda\right)\right].\label{eq:entropy-2}
\end{equation}
Let $\randl=\left(\randl_{v}\right)_{v\in\range n}$ be a sequence
of independent random variables, where each $\randl_{v}$ has the
uniform distribution in $\left[0,1\right]$. It follows from \ref{eq:entropy-1}
and \ref{eq:entropy-2} that
\[
H\left(\randM\right)\le\sum_{v\in\range n}\E\left[\log\randR_{v}\left(\randl\right)\right].
\]
Next, for any $M\in\mathcal{M}$ and $\lambda_{v}\in\left[0,1\right]$,
let
\[
R_{v}^{M,\lambda_{v}}=\E\left[\randR_{v}\left(\randl\right)\cond\randM=M,\,\randl_{v}=\lambda_{v},\,\randl_{v'}<\randl_{v}\text{ for all }v'\in\randM_{v}\backslash\left\{ v\right\} \right].
\]
Now, there are $\left(n-1\right)/2$ hyperedges in $S$ containing
$v$, and for each such hyperedge $e=\left\{ x,y,v\right\} $ other
than $M_{v}$, note that $M_{x}\ne M_{y}$ (because $e$ and $M_{x}$
are different hyperedges of a Steiner triple system and can therefore
intersect in at most one vertex). So, $\left|\left(\bigcup_{z\in e}M_{z}\right)\backslash M_{v}\right|=6$
and by linearity of expectation,
\[
R_{v}^{M,\lambda_{v}}=1+\left(\left(n-1\right)/2-1\right)\lambda_{v}^{6}.
\]
By Jensen's inequality,
\[
\E\left[\log\randR_{v}\left(\randl\right)\cond\randM=M,\,\randl_{v}=\lambda_{v},\,\randl_{v'}<\randl_{v}\text{ for all }v'\in\randM_{v}\backslash\left\{ v\right\} \right]\le\log R_{v}^{M,\lambda_{v}},
\]
and
\[
\Pr\left(\randl_{v'}<\randl_{v}\text{ for all }v'\in\randM_{v}\backslash\left\{ v\right\} \cond\randl_{v}=\lambda_{v}\right)=\lambda_{v}^{2},
\]
so
\[
\E\left[\log\randR_{v}\left(\randl\right)\cond\randM=M,\,\randl_{v}=\lambda_{v}\right]\le\lambda_{v}^{2}\log R_{v}^{M,\lambda_{v}}+\left(1-\lambda_{v}^{2}\right)\log1=\lambda_{v}^{2}\log R_{v}^{M,\lambda_{v}}.
\]
It follows that
\begin{align*}
\E\left[\log\randR_{v}\left(\randl\right)\cond\randM=M\right] & \le\E\left[\randl_{v}^{2}\log R_{v}^{M,\randl_{v}}\right]\\
 & =\int_{0}^{1}\lambda_{v}^{2}\log\left(1+\left(\left(n-1\right)/2-1\right)\lambda_{v}^{6}\right)\d\lambda_{v}\\
 & =\frac{1}{3}\left(\log\left(n/2\right)-2\right)+o\left(1\right),
\end{align*}
using \ref{eq:complicated-integral} (from the proof of \ref{thm:number-of-completions-upper}
in \ref{sec:completion-count}). We conclude that
\begin{align*}
\log\,\left|\mathcal{M}\right| & =H\left(\randM\right)\\
 & \le\sum_{v\in\range n}\E\left[\log\randR_{v}\left(\randl\right)\right]\\
 & \le\frac{n}{3}\left(\log\left(n/2\right)-2+o\left(1\right)\right),
\end{align*}
which is equivalent to the theorem statement.
\end{proof}

\section{Latin squares\label{sec:latin}}

In this section we sketch how one should adapt the methods in this
paper to prove \ref{thm:latin}.

A partial Steiner triple system is a collection of edge-disjoint triangles
in $K_{n}$, whereas a partial Latin square is a collection of edge-disjoint
triangles in the complete tripartite graph $K_{n,n,n}$. Let $V=V_{1}\sqcup V_{2}\sqcup V_{3}$
be the tripartition of $K_{n,n,n}$. We say a subgraph $G\subseteq K_{n,n,n}$
with $m$ edges between each pair of parts is $\left(\varepsilon,h\right)$-quasirandom
if for each $i\in\left\{ 1,2,3\right\} $, every set $A\subseteq V\backslash V_{i}$
with $\left|A\right|\le h$ has $\left(1\pm\varepsilon\right)\left(m/n^{2}\right)^{\left|A\right|}n$
common neighbours in $V_{i}$. The following result is a special case
of \cite[Theorem~1.5]{Kee18}.
\begin{thm}
\label{conj:keevash-latin}There are $h\in\NN$, $\varepsilon_{0},a\in\left(0,1\right)$
and $n_{0},\ell\in\NN$ such that the following holds. Suppose $n\ge n_{0}$,
$m/n^{2}\ge n^{-a}$ and $\varepsilon\le\varepsilon_{0}\left(m/n^{2}\right){}^{\ell}$,
and suppose that $G\subseteq K_{n,n,n}$ is an $\left(\varepsilon,h\right)$-quasirandom
graph with $m$ edges between each pair of parts, that arises as the
set of edges of $K_{n,n,n}$ not covered by the triples of some partial
Latin square. Then the edges of $G$ can be decomposed into triangles.
\end{thm}

With our new notion of quasirandomness and \ref{conj:keevash-latin}
playing the role of \ref{thm:keevash} we can then adapt the proofs
in \ref{sec:completion-count,sec:random-triangle-removal} in a straightforward
manner to prove the obvious Latin squares counterpart to \ref{lem:num-extensions}.
This allows us to prove the Latin squares counterpart to \ref{lem:triangle-removal-transfer}.

Now we outline what should be adapted from the arguments in \ref{sec:absorbers}
for the Latin squares case. The definition of an absorber can remain
the same, noting that the hypergraph in \ref{def:absorber} is tripartite
(and the tripartition can be chosen to have $x,y,z$ in different
parts). The definition of a resilient template should be adapted slightly:
we define a resilient template to be a tripartite hypergraph $H$
(with tripartition $V\left(H\right)=V_{1}\left(H\right)\sqcup V_{2}\left(H\right)\sqcup V_{3}\left(H\right)$,
say) with a flexible set $Z$, such that each $Z_{i}=V_{i}\left(H\right)\cap Z$
has the same size, and such that if half the vertices of each $Z_{i}$
are removed, then the remaining hypergraph has a perfect matching.
To prove a counterpart of \ref{lem:resilient-template} we can just
use three vertex-disjoint  copies of the tripartite hypergraph in
the proof of \ref{lem:resilient-template} (one copy for each $Z_{i}$).
The counterpart of \ref{lem:PM-via-absorbers} is as follows (with
virtually the same proof).
\begin{lem}
\label{lem:absorbers-latin}Consider an order-$n$ Latin square $L$
(with tripartition $V=V_{1}\sqcup V_{2}\sqcup V_{3}$) satisfying
the following properties for some $\delta=\delta\left(n\right)=o\left(1/\log n\right)$
and fixed $\beta>0$.
\begin{enumerate}
\item There is an absorbing structure $H$ in $L$ with at most $\delta n$
vertices and a flexible set $Z$ such that each $Z_{i}=V_{i}\cap Z=2\floor{\delta^{2}n}$.
\item For at most $\delta n$ of the vertices $v\in V_{1}$, we have $\left|\left\{ \left(x,y\right)\in Z_{2}\times Z_{3}:\left(v,x,y\right)\in E\left(L\right)\right\} \right|<6\delta^{5}n$,
and the analogous statements hold for degrees of $v\in V_{2}$ and
$v\in V_{3}$ into $Z_{1}\times Z_{3}$ and $Z_{1}\times Z_{2}$ respectively.
\item For any choice of $W_{i}\subseteq V_{i}$ such that each $\left|W_{i}\right|\ge\delta^{5}n$,
there are at least $\left(1-\beta\right)\left|W_{1}\right|\left|W_{2}\right|\left|W_{3}\right|/n$
hyperedges in $W_{1}\times W_{2}\times W_{3}$.
\end{enumerate}
Then $L$ has 
\[
\left(\frac{n}{e^{2}}\left(1-\beta-o\left(1\right)\right)\right)^{n}
\]
transversals.
\end{lem}

One can then use \ref{lem:absorbers-latin} to prove \ref{thm:latin}
in basically the same way as the proof of \ref{thm:PM-in-STS} in
\ref{sec:absorbers-in-TRP}. We remark that in \ref{lem:new-quasirandom}
we used an implication between certain different notions of graph
quasirandomness; to straightforwardly adapt this lemma to the Latin
squares case we would need an analogous implication between notions
of ``bipartite quasirandomness''. Such an implication is well-known
to hold (see for example \cite[Exercise 9.10]{AS04}).

\section{Concluding remarks\label{sec:concluding}}

In this paper we introduced a new method for analysing random Steiner
triple systems, and we used it to prove that almost all Steiner triple
systems have many perfect matchings. There are many interesting open
questions that remain.
\begin{itemize}
\item We believe the most interesting problem that seems approachable by
our methods is to prove that almost all Steiner triple systems (and
Latin squares) can be \emph{decomposed} into disjoint perfect matchings
(transversals). Since the first version of this paper, together with
Ferber we~\cite{FK19} proved an approximate version of this fact:
namely, almost all Steiner triple systems have $\left(1-o\left(1\right)\right)n/2$
disjoint perfect matchings. As mentioned in the introduction, a Steiner
triple system that can be decomposed into perfect matchings is called
a Kirkman triple system, and even proving the existence of Kirkman
triple systems was an important breakthrough. For Latin squares, the
property of being decomposable into transversals is equivalent to
the important property of having an \emph{orthogonal mate}, which
has a long history dating back to Euler. More details can be found
in \cite{Wan11}.
\item A \emph{$\left(q,r,\lambda\right)$-design} ($q>r$) of order $n$
is a $q$-uniform hypergraph on the vertex set $\range n$ such that
every $r$-set of vertices is included in exactly $\lambda$ hyperedges.
A $\left(q,r\right)$\emph{-Steiner system} is a $\left(q,r,1\right)$-design
(so, a Steiner triple system is a $\left(3,2,1\right)$-design or
equivalently a $\left(3,2\right)$-Steiner system, and a $d$-regular
graph is a $\left(2,1,d\right)$-design). The methods in \ref{sec:random-STS}
generalise to $\left(q,r,\lambda\right)$-designs with mainly notational
changes. Note that a $3$-uniform perfect matching is actually a $\left(3,1\right)$-Steiner
system, so as a sweeping generalization of \ref{thm:PM-in-STS} we
might ask for which $r'\le r$ and $\lambda'\le\lambda$ do $\left(q,r,\lambda\right)$-designs
typically contain spanning $\left(q,r',\lambda'\right)$-designs of
the same order. We note that in the case of regular graphs a much
stronger phenomenon occurs: there is a sense in which a random $\left(d_{1}+d_{2}\right)$-regular
graph is ``asymptotically the same'' as a random $d_{1}$-regular
graph combined with a random $d_{2}$-regular graph (see \cite[Section~9.5]{JLR00}).
\item Another interesting question about random Steiner triple systems is
whether they contain Steiner triple subsystems on fewer vertices.
McKay and Wanless \cite{MW99} proved that almost all Latin squares
have many small Latin subsquares (see also \cite{KS16}), but it was
conjectured by Quackenbush \cite{Qua80} that most Steiner triple
systems do not have proper subsystems. It seems unlikely that the
methods in this paper will be able to prove or disprove this conjecture
without substantial new ideas. Actually, by consideration of the random
3-uniform hypergraph $\Gnp n{1/n}$ we suspect the expected number
of 7-vertex Steiner triple subsystems (Fano planes) in a random Steiner
triple system is $\Theta\left(1\right)$, and that the distribution
of this number is asymptotically Poisson.
\item We could ask more generally about containment and enumeration of subgraphs.
Which hypergraphs $H$ appear a.a.s. in a random Steiner triple system?
Can we show that for all such $H$ the number of copies of $H$ is
concentrated? The methods in this paper can probably be used to prove
a lower bound for the number of copies of $H$ when every subgraph
of $H$ has at least 2 more vertices than hyperedges, but due to the
``infamous upper tail'' issue (see \cite{JR02}), an upper bound
for the number of copies of $H$ is likely to be more difficult.
\item One of the most fundamental properties of random graphs and hypergraphs
is that they have low \emph{discrepancy}, meaning that every sufficiently
large subset of vertices has about the expected number of (hyper)edges.
In \ref{subsec:density} we effectively proved a very weak one-sided
discrepancy bound, but it is not clear how to use our methods to reach
anywhere near optimal discrepancy. See \cite{LL15} for some theorems
and conjectures about discrepancy of Latin squares, many of which
have natural counterparts for Steiner triple systems.
\end{itemize}

\end{document}